%% file: main.tex
\documentclass{template}

\usepackage{graphicx} 

\usepackage[style=alphabetic]{biblatex}
\usepackage{caption} 
\usepackage[ruled,vlined]{algorithm2e}

\usepackage{amsfonts}

\usepackage{xr}

\usepackage{bm}
\renewcommand{\t}{\tilde}

\usepackage{mathtools}

\newcommand{\vc}{\textrm{vec}}
\newcommand{\tr}{\textrm{tr}}

\usepackage{xcolor} 

\newcommand{\bmu}{\bm{\mu}}
\newcommand{\bdelta}{\bm{\delta}}
\newcommand{\bmv}{\bm{v}}
\newcommand{\bmw}{\bm{w}}
\newcommand{\bw}{\bm{w}}
\newcommand{\bv}{\bm{v}}

\newcommand{\be}{\bm{e}}
\newcommand{\bme}{\bm{e}}
\newcommand{\bx}{\bm{x}}

\renewcommand{\bb}{\bm{b}}

\newcommand{\by}{\bm{y}}

\usepackage{amsthm}

\newtheorem{corollary}[theorem]{Corollary}
\newtheorem{proposition}[theorem]{Proposition}
\newtheorem{definition}[theorem]{Definition}

\title{Strategic Negotiations in Endogenous Network Formation}

\author{
    Akhil Jalan \\
    UT Austin \\
    \texttt{akhiljalan@utexas.edu}
    \and
    Deepayan Chakrabarti \\
    UT Austin \\
    \texttt{deepay@utexas.edu}
}

\usepackage{subcaption}
\begin{document}
\maketitle
\begin{abstract}
In network formation games, agents form edges with each other to maximize their utility.
Each agent's utility depends on its private beliefs and its edges in the network.
Strategic agents can misrepresent their beliefs to get a better resulting network.
Most prior works in this area consider honest agents or a single strategic agent.
Instead, we propose a model where any subset of agents can be strategic.
We provide an efficient algorithm for finding the set of Nash equilibria, if any exist, and certify their nonexistence otherwise.
We also show that when several strategic agents are present, their utilities can increase or decrease compared to when they are all honest.
Small changes in the inter-agent correlations can cause such shifts. 
In contrast, the simpler one-strategic-agent setting explored in the literature lacks such complex patterns.
Finally, we develop an algorithm by which new agents can learn the information needed for strategic behavior.
Our algorithm works even when the (unknown) strategic agents deviate from the Nash-optimal strategies.
We verify these results on both simulated networks and a real-world dataset on international trade.
\end{abstract}

\input{body}
\printbibliography

\appendix

\input{appendix}

\end{document}

%% file: body.tex


\section{Introduction}\label{sec:intro}
\input{intro}

Next, we discuss notation and present a background on network formation without strategic agents.

\medskip\noindent
\textbf{Notation.} We will use lowercase letters $a, b, c, \gamma_i$ to denote scalars, boldface letters $\bmu_i, \bm{w}$ 
to denote vectors, and uppercase letters $A, B, \Sigma$ to denote matrices. The vectors $\bme_1, \ldots, \bme_n$ 
denote the standard basis in $\RR^n$, and $I_n$ is the $n\times n$ identity matrix. 
We use $\bv_{i;j}$ to refer to the $j^{th}$ component of the vector $\bv_i$. We denote the inner product $\la \bm{v}, \bm{w} \ra := \bm{v}^T \bm{w}$. 
We use $A\succ 0$ to denote that $A$ is positive definite.
If $A \in \RR^{m \times n}, B \in \RR^{p \times q}$ then $A \otimes B \in \RR^{mp \times nq}$ denotes their tensor  product: $(A \otimes B)_{ij, k\ell} = A_{ik}B_{j \ell}$.
For an appropriate matrix $M$, $\tr(M)$ calculates its trace and $\vc(M)$ vectorizes $M$ by stacking its columns into a single vector.
For an integer $r \geq 1$, we use $[r]$ to denote $[r] \defeq \{1, 2, \dots, r\}$. 


\medskip\noindent
{\bf Background: network formation without strategy.}
We use a network model with side-payments between agents~\cite{jackson-wolinsky-2003} and mean-variance utility, which is a widely used model of risk-aware utility~\cite{harrison2009minimum,li2014mean,simaan2014opportunity,zhang2021mean, ma-2023}. This network model has been shown to provide closed-form solutions for truthful network formation~\cite{jalan-2023}. We summarize this model below.

Let $W=W^T\in\RR^{n\times n}$ denote an undirected weighted network of contracts between $n$ agents, with $W_{ij}$ being the size of the contract between $i$ and $j$ and $W_{ii}$ representing self-investment. 
A negative contract $W_{ij} < 0$ is valid and represents a reversed version of a positive contract; for example, in a derivative contract, $W_{ij}<0$ swaps the roles of the long and short position holders. 
During contract negotiations, agent $i$ can pay $P_{ji}$ per unit contract to agent $j$ to get $j$ to agree to the contract size.
Since payments are zero-sum, $P^T = -P$. 
The contracts size and payments $(W, P)$ together give the network. At $(W, P)$, agent $i$ has contracts $\bmw_i := W \be_i$. Agent $i$ wants to optimize the utility of their contracts and believes that contracts have mean return ${\bm\mu}_i \in \RR^n$ and covariance $\Sigma \succ 0$. Moreover, they have a risk-aversion parameter $\gamma_i > 0$. Their utility is then: 
\begin{align}
\text{agent $i$'s utility } g_i(W, P) &:= \bmw_i^T (\bmu_i - P \bme_i) - \gamma_i \cdot \bmw_i^T \Sigma \bmw_i.
\label{eq:utility}
\end{align}

Note that beliefs do not have to be accurate or follow a particular distribution. 



%
%


\begin{definition}[Stable point]
A feasible $(W, P)$ is stable if each agent achieves its maximum possible utility given prices $P$:
\begin{align*}
g_i(W, P) = \max\limits_{(W^\prime, P): W^\prime = {W^\prime}^T, P^T = -P} g_i(W', P) \quad \forall i\in [n].
\end{align*}
\end{definition}

\begin{theorem}[Stable network without strategy~\cite{jalan-2023}]
\label{thrm:sylvester-eqns}
Let $M$ be such that $M \bme_i=\bmu_i$. Let $\Gamma$ be a diagonal matrix with $\Gamma_{ii}=\gamma_i$.
Note that $\Gamma \succ 0$ and $\Sigma \succ 0$. There exists a unique stable point $(W, P)$: 
\begin{align*}
\vc(W) &= \frac 1 2 (\Gamma\otimes \Sigma + \Sigma \otimes \Gamma)^{-1}\vc(M+M^T), \\
\vc(P) &= \big((\Gamma^{-1} \otimes \Sigma^{-1} + \Sigma^{-1} \otimes \Gamma^{-1})^{-1} \\
&\cdot \vc(\Sigma^{-1} M \Gamma^{-1} - \Gamma^{-1} M^T \Sigma^{-1})\big).
\end{align*}
Furthermore, agents can efficiently find the stable point through honest pairwise negotiations. 
\end{theorem}

\section{Strategic Negotiations}\label{sec:deception}

We now formalize the contract negotiation process. 

\begin{definition}[Our Model of Strategic Contract Negotiation $(M, \Gamma, \Sigma, S)$]
\label{def:negotiation}
There is a set $S \subseteq [n]$ of strategic agents who know the $(M, \Gamma, \Sigma)$ defined in Theorem~\ref{thrm:sylvester-eqns}.
An honest agent $i\notin S$ only knows $(\bm{\mu}_i, \gamma_i, \Sigma)$.
The contract negotiation is a two-stage process:
\begin{enumerate}
  \item {\bf Strategy Phase}: Each strategic agent $k \in S$ independently and privately chooses a negotiating position $\bmu_k^\prime \in \RR^n$. For honest agents $k \not \in S$, $\bmu_k^\prime = \bmu_k$. Let $M^\prime$ be a matrix whose $i^{th}$ column is $\bmu_i^\prime$. 
  \item {\bf Contract Formation Phase}: The network is formed as if every agent's negotiating position was their true belief. 
  Specifically, 
  $(W,P)$ is formed according to Theorem~\ref{thrm:sylvester-eqns} with $(M^\prime, \Sigma, \Gamma)$.
\end{enumerate}
The network $(W,P)$, and the true beliefs $(M, \Gamma, \Sigma)$ determine each agent's utility (Eq.~\eqref{eq:utility}).
\end{definition}


The above definition assumes that strategic agents know the true beliefs $(M, \Gamma, \Sigma)$ (Definition~\ref{def:negotiation}). 
Our approach generalizes to the case where agents have a distribution over $M$, and each agent aims to maximize its expected utility.
For ease of exposition, we focus on the fixed $M$ setting here, with the general setting deferred to the Appendix.
Also, we do not consider strategic choices for the risk aversion $\gamma_i$ and covariance matrix $\Sigma$.
The former is typically similar for all agents~\cite{paravisini-2017}, while the latter is often known from public sources such as credit rating agencies~\cite{white2010markets}.


We first prove a general result that characterizes an agent's utility given {\em arbitrary} negotiating positions for all other agents.
It also shows that no agent can gain unbounded utility by being strategic. All proofs are deferred to the Appendix.

\begin{theorem}[Concave utility given others' choices]\label{thrm:lying-bounded}
Given $M, \Gamma, \Sigma$ and any set of negotiating positions $\{\bmu_i^\prime; i\neq k\}$ for all agents except $k$, agent $k$'s utility is a quadratic function of $\bmu_k^\prime$ with a negative definite Hessian.
\end{theorem}
\begin{corollary}[Strategy yields bounded utility]\label{corr:lying-bounded}
No choice of negotiating position lets agent~$k$ achieve unbounded utility, even if agent $k$ has full information about other agents' beliefs and choices.
\end{corollary}

A key ingredient of our proof is that risk $\bm{w}^T \Sigma \bm{w}$ scales quadratically in $\norm \bw \norm$, while reward only scales linearly. Hence the quadratic form of the mean-variance utility enables a realistic analysis. 


In strategic contract negotiations, all agents make strategic choices independently and cannot adapt their strategy to the others' choices ex post.
They can choose a strategy ex ante based on a Nash equilibrium, defined below.

\begin{definition}
A Nash Equilibrium for a  Strategic Contract Negotiation $(M, \Gamma, \Sigma, S)$ is a matrix $\Delta \in \RR^{n\times n}$ with the following property.
For each $k \in S$, if all other strategic agents $j \in S$ choose negotiating position $(M + \Delta) \be_j$, then agent $k$ gains the highest utility by choosing negotiating position $(M + \Delta) \be_k$ in the Strategy Phase.
\end{definition}
Notice that $\Delta = M^\prime - M$, so for a fixed $M$ the negotiating positions are determined by the columns of $\Delta$.

We allow the matrix $\Delta$ to be random, corresponding to mixed strategies. However, we will see that even if others play mixed strategies, the optimal choice for an agent is to play a pure strategy (Theorem~\ref{thrm:indiv-optimal}).

\SetKwComment{Comment}{/* }{ */}
\begin{algorithm}[t]
	\KwIn{$(M, \Gamma, \Sigma)$ as in Definition~\ref{def:negotiation}, strategic agent set $S \subseteq [n]$.}
	\KwOut{Nash equilibria set $\mathcal{E}$.} 
	$T, K, L, \{\bm{y}_k\} \gets$ as in Definition~\ref{defn:LKT}. \\
	$T_S \gets$ submatrix of $T$ with blocks $\{T^{(i,j)}: i,j\in S\}$\\
	$\bm{y}_S \gets$ concatenation of $\{\bm{y}_k: k \in S\}$\\
	$\mathcal{F} \gets \{X \in \RR^{n \times |S|}: T_S \vc(X) = \bm{y}_S\}$\\
    $\mathcal{E} \gets 
	\bigg\{\Delta \in \RR^{n \times n}: \bigg\{\begin{array}{c}\Delta|_S=X,\\ \Delta|_{[n]\setminus S}=0\end{array}\bigg\}, X\in\mathcal{F}\bigg\}$\\
	\Return{$\mathcal{E}$ if $|\mathcal{E}|>0$ else ``No Nash Equilibrium''}
	\caption{Nash Equilibria Computation}
\label{alg:nash}
\end{algorithm}

To give an explicit solution for optimal negotiating positions, we require the following definitions. 

\begin{definition}
We define the commutator matrix $\Pi: \RR^{n^2} \to \RR^{n^2}$ and the projection matrices $\Pi_k: \RR^{n^2} \to \RR^n$ such that $\Pi \vc(X) = \vc(X^T)$ and $\Pi_k \vc(X) = X \be_k$ for all $X \in \RR^{n \times n}$.
For any $Z\in\RR^{n^2}\to\RR^{n^2}$, we define $Z^{(p,q)}$ as the $n\times n$ block at position $(p,q)$.
\label{defn:commutation-matrix}
\end{definition}


\begin{defn}
For $(M, \Gamma, \Sigma)$ as in Definition~\ref{def:negotiation}, we define the following:
\begin{align*}
K &= (\Gamma \otimes \Sigma + \Sigma \otimes \Gamma),\\
L &= \frac 1 2 (K^{-1} + K^{-1} \Pi),\\
T^{(k,j)} &= 
\begin{cases}
L^{(k, k)} + (L^{(k, k)})^T - 2 \gamma_k (L^{(k, k)})^T \Sigma L^{(k, k)} & k = j \\
(I - 2 \gamma_k (L^{(k,k)})^T \Sigma) L^{(k, j)} & k \neq j 
\end{cases}\\
\bm{y}_k &= \frac 1 2 (2 \gamma_k (L^{(k, k)})^T \Sigma - I) \Pi_k K^{-1} \vc(M + M^T).
\end{align*}

\label{defn:LKT}
\end{defn}
Note that $K^{-1}$ exists because $\Sigma, \Gamma \succ 0$.


We can now fully characterize the optimal negotiating position of an agent who is uncertain about the positions of her counterparties.

\begin{theorem}
Suppose a strategic agent $k \in [n]$ knows $S$, and has a distribution $\mathcal{D}_k$ over the negotiation positions $\{\bmu_{i}^\prime: i \in S\setminus\{k\}\}$ with finite first and second moments.
Define $\bdelta_i:=\bmu_i^\prime-\bmu_i$.
The negotiating position $\bmu_k^\prime$ (or, equivalently, the $\bdelta_k$) that optimizes $k$'s expected utility with respect to $\mathcal{D}_k$ is given by the solution(s) to the following linear system, if any exist. 
\begin{align}
T^{(k,k)} \bdelta_k 
+ \sum\limits_{j \in S: j \neq k} T^{(k,j)} \EE\limits_{\mathcal{D}_k} \bdelta_j 
&= \bm{y}_k
\label{eq:opt}
\end{align}
\label{thrm:indiv-optimal}
\end{theorem}


Thus, the optimal negotiating position of an agent $k \in S$ is a fixed $\bdelta_k^* \in \RR^n$ that solves a deterministic linear system. In a Nash equilibrium, every strategic agent solves their corresponding equation.

\begin{corollary}[Nash Equilibria] 
The Nash equilibrium corresponds to solving the system of $n\abs{S}$ linear equations in the fixed vectors $\{\bdelta_i\mid i\in S\}$ given by taking Eq.~\ref{eq:opt} for each $k \in S$. 
\label{cor:nash-system}
\end{corollary}

\begin{corollary}[All Equilibria are Pure]
All Nash equilibria are pure-strategy Nash equilibria. 
\end{corollary}


Algorithm~\ref{alg:nash} explicitly describes how a strategic agent can solve for the equilibrium $\Delta$.

In summary, we see that strategic agents can negotiate optimally and find Nash equilibria.
This motivates two questions that we will address in turn. First, do strategic agents always get greater utility from negotiating strategically (Section~\ref{sec:welfare})? 
Second, strategic agents need to know the matrix $M$ and the set of other strategic agents $S$. Can agents learn these from observing the network (Section~\ref{sec:learning})?

\section{Outcomes for Multiple Strategic Actors}\label{sec:welfare}

The motivation for negotiating strategically, rather than honestly, is that an agent might achieve better terms for their contracts and hence more utility. We depart from previous works on network formation by studying strategic behavior for {\em any} set $S \subset [n]$ of strategic agents, rather than $\abs{S} = 1$. 

If $\abs{S} = 1$, it is clear from Theorem~\ref{thrm:lying-bounded} that the lone strategic agent can do no worse by negotiating strategically. 
However, for $\abs{S} > 1$, the situation is less clear. There are three possibilities: 
\begin{enumerate}
	\item All strategic agents in $S$ are better off than if they had all negotiated honestly. 
	\item Some members of $S$ are worse off. 
	\item All members of $S$ are worse off. 
\end{enumerate}

In this section, we show that all three possibilities can occur, even in a simple network with $n = 3$ agents. 
Outcome (3) is especially interesting, as every strategic agent would be better off if all of them were honest.
However, our model does not allow agents to coordinate.
Hence, they are stuck in a lose-lose Nash equilibrium, akin to the Prisoner's Dilemma.

We now introduce the following example network. More details and other examples are presented in the Appendix. 

\begin{example}[Two Hedge Funds, One Investor (Figure~\ref{fig:both-figures})]
For $m, a \in \RR$ and $\rho \in (-1,1)$, define:
\begin{align}
M = \begin{bmatrix} 0 & a & a \\ m & 0 & 0 \\ m & 0 & 0 \end{bmatrix}, 
\quad \Sigma = \begin{bmatrix} 1 & 0 & 0 \\
0 & 1 & \rho \\
0 & \rho & 1
\end{bmatrix},
\quad \Gamma = I.
\label{eq:3p}
\end{align}
The first column corresponds to the investor, and the others to the hedge funds.
Under this setting, the hedge funds do not want to trade with each other, and none of the agents want to self-invest.
Also, the hedge funds are correlated with each other (via $\rho$), and uncorrelated with the investor.
\label{ex:investor}   
\end{example}

We will consider two cases: (a) the investor is honest and the hedge funds are strategic ($S = \{2,3\}$), and (b) all agents are strategic ($S = \{1, 2, 3\}$).


\begin{figure}[!htbp]
    \centering
    \begin{subfigure}{\linewidth} 
        \centering
        \includegraphics[width=0.4\linewidth]{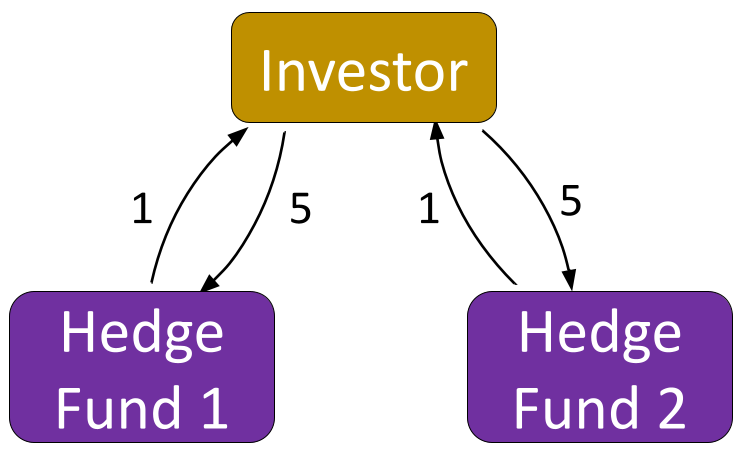}
        \caption{Network diagram with entries of entries of $M$ marked on edges.}
        \label{fig:enter-label}
    \end{subfigure}
    
    
    \begin{subfigure}{\linewidth} 
        \centering
        \includegraphics[width=0.99\linewidth]{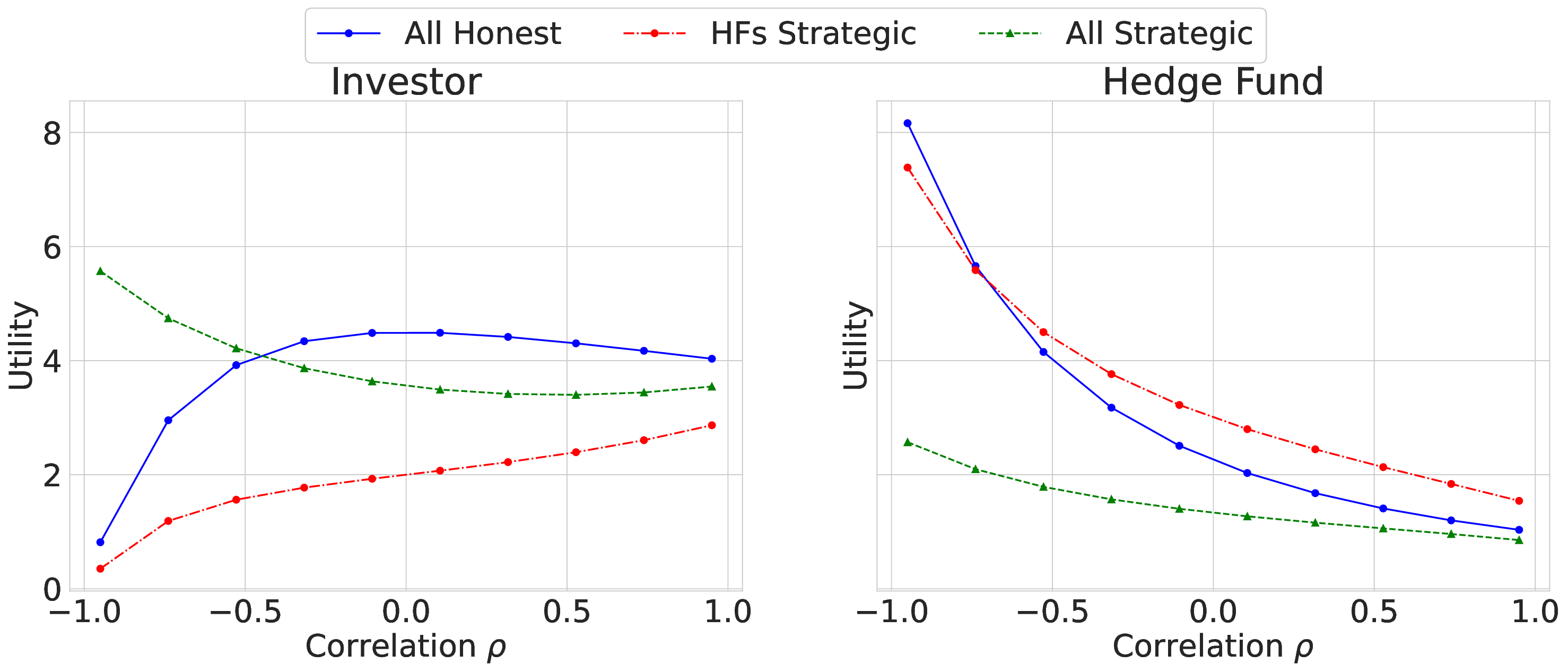}
        \caption{Utilities of the Investor and a single Fund under Proposition~\ref{prop:3p-insights}. Note that the two Funds achieve the same utility.}
        \label{fig:insights-3p}
    \end{subfigure}
    
    \caption{{\em Network with three agents:}
    (a) An investor trades with two hedge funds, with the investor gaining $5$ per unit contract while the hedge funds gain $1$.
    (b) We show the utility for the investor (left) and either hedge fund (right) for various strategic behaviors.
    The investor has low utility when she is honest, but is better off than the hedge funds when she is strategic.
      }
    \label{fig:both-figures}
\end{figure}

\begin{proposition}
\label{prop:3p-insights}
Consider the setting of Eq.~\ref{eq:3p}, and suppose the strategic agents can only modify the non-zero entries in their column of $M$ for their negotiating positions.
Define
\begin{align*}
\nu &= \frac{1}{2}\left(\frac{1}{2-\rho} + \frac{1}{2+\rho}\right), &
\eta &= \frac{1}{2}\left(\frac{1}{2-\rho} - \frac{1}{2+\rho}\right), \\
\zeta &= \frac{\nu-\eta}{\nu + (\nu-\eta)(1-\nu)}.
\end{align*}
The matrix $M^\prime$ of optimal negotiating positions is as follows:
\begin{enumerate}
\item {\em Honest investor and strategic hedge funds ($S = \{2,3\}$):} 
\begin{align*}
M^\prime_{21}=M^\prime_{31}&= m, &
M^\prime_{12}=M^\prime_{13} &= \frac{a\nu - m (1-\nu)(\nu-\eta)}{\nu + (1-\nu)(\nu-\eta)}.
\end{align*}

\item {\em All agents strategic ($S = [n]$):} 
\begin{align*}
M^\prime_{21}=M^\prime_{31} &= \frac{m - a \zeta}{1 + \zeta}, \\
M^\prime_{12}=M^\prime_{13} &= \frac{a\nu - M^\prime_{21} (1-\nu)(\nu-\eta)}{\nu + (1-\nu)(\nu-\eta)}.
\end{align*}
\end{enumerate}
\end{proposition}


From Proposition~\ref{prop:3p-insights} we can compute the utilities at equilibrium for any particular choice of $M$ and $\Sigma$. Figure~\ref{fig:insights-3p} shows the utilities for a specific $M$ and varying $\rho$. We observe the following.

\medskip\noindent
{\bf When only the hedge funds are strategic, they can be both better off or both worse off (Outcomes~1 and~3).}
The specifics depend on the perceived correlation $\rho$. As $\rho \approx -1$, the hedge funds are worse off being strategic than if they were both honest (Figure~\ref{fig:insights-3p}). This is due to the hedging behavior of the investor. 

As $\rho \to -1$, the investor wishes to invest almost equally in both funds to reduce her overall risk. But the hedge funds only form one contract each. Since they cannot hedge their risk, they prefer much smaller contracts than the investor. If both funds are honest, they can negotiate contract sizes to match their risk preference. However, if both are strategic, each fund worries about its competitor. So, both funds end up taking on more risk than they would prefer, and are worse off. 

On the other hand, for $\rho \gg -1$, the investor will not seek such large contracts, since she cannot hedge as well. So the hedge funds are both better off being strategic.

\medskip\noindent
{\bf When all agents are strategic ($S = [n]$), the investor can be better off while the funds are both worse off (Outcome~2).} When $S = [n]$, there is no setting in which all agents are better off. However, the investor is better off as $\rho \to -1$. As before, the funds are worse off because they are forced to take large contracts. When the investor is also strategic, she can force the funds to compete for her investment and obtain better terms from both. She will obtain large contracts with both, which enable low risk due to hedging and have better terms than if she was honest.

\input{learning}

\section{Experiments}\label{sec:experiments}

We show experiments for learning the network parameters on a simulated dataset, and then test the effects of strategic negotiations on the OECD international trade network.

\subsection{Learning Experiments}
We validate our learning approach on networks where agent $i$ has $d$-dimensional features $\bx_i$ sampled independently from $\text{Dirichlet}(1/d, 1/d, \ldots, 1/d)$, and $M$ has a bilinear form with a random symmetric matrix $B\in\RR^{d\times d}$ with upper triangular entries $N(5, 1)$ (Definition~\ref{def:mmsb}).
Then, we sample $S \subseteq [n]$ uniformly from all subsets of a certain size, and compute the stable network $W^\prime$ with Algorithm~\ref{alg:nash}. All experiments use $n = 100$ agents and $\beta := \frac{2 \abs{S} n - \abs{S}^2}{n^2}$. See the Appendix for full experimental details. 

Figures~\ref{fig:regression_3wid} and \ref{fig:acc_3wide} show the accuracy of the recovered matrix $\hat{B}$ and the strategic subset of agents $\hat{S}$, respectively.
We find that Algorithm~\ref{alg:learning} performs well for a broad range of parameter settings. Unsurprisingly, it is best for small $d$ and $\abs{S}$. We find that the regression error is low even with a large number of strategic actors (Figure~\ref{fig:regression_3wid}). This suggests that the condition of $\abs{S} \leq Cn$ in Proposition~\ref{prop:b-recovery}, which is required to handle the worst-case $S$, may be relaxed if we are willing to accept an average-case guarantee. 



\begin{figure}[!htbp]
	\centering
    \includegraphics[width=0.6\textwidth]{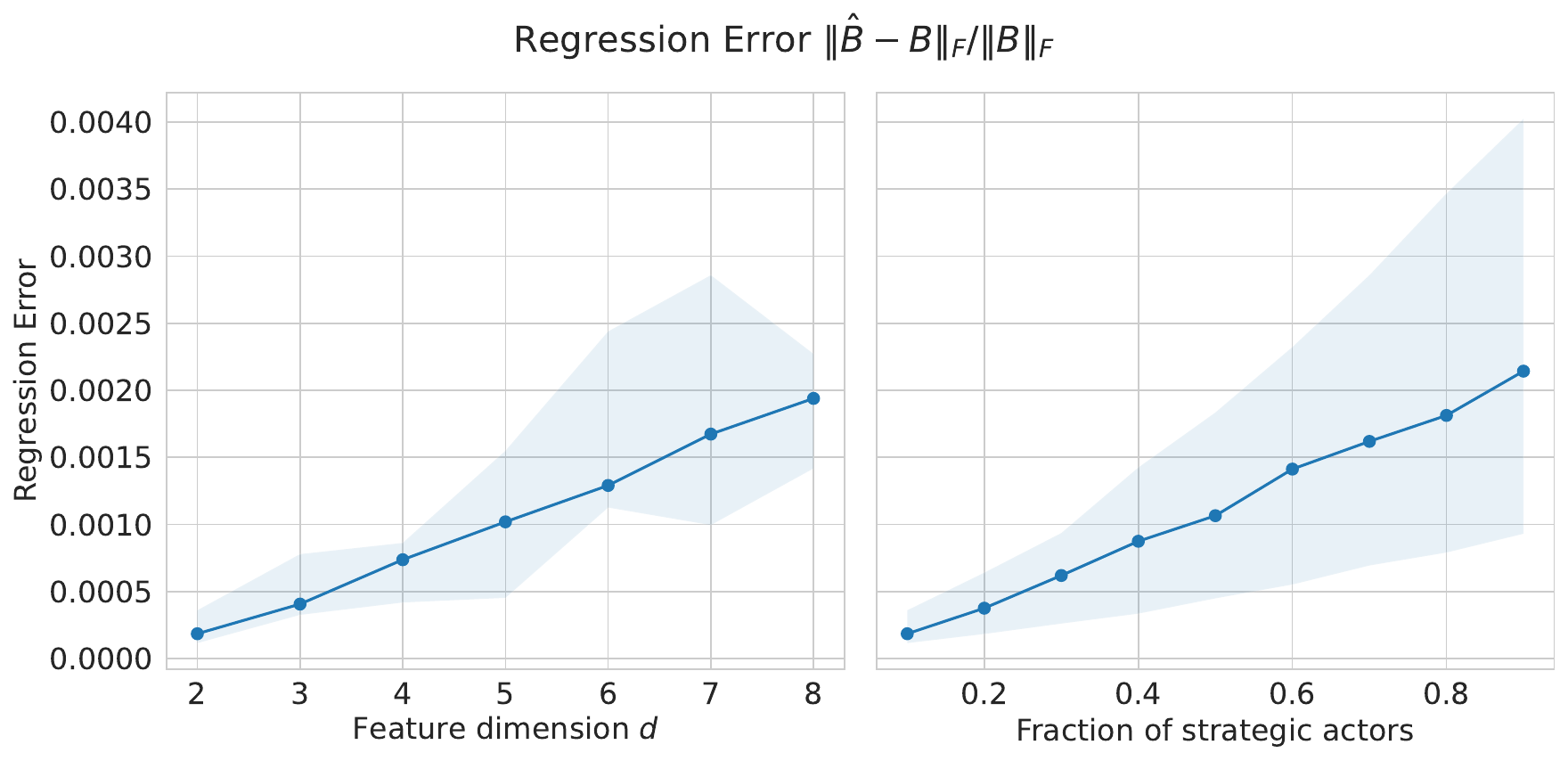}
    \caption{Normalized regression error for $\hat B$ estimation, with shaded regions denoting $[10,90]$-percentile outcomes across $10$ independent trials. Left: $\abs{S} = 0.1\times n$. Right: $d = 2$.}
    \label{fig:regression_3wid}
\end{figure}

\begin{figure}[!htbp]
	\centering
    \includegraphics[width=0.6\textwidth]{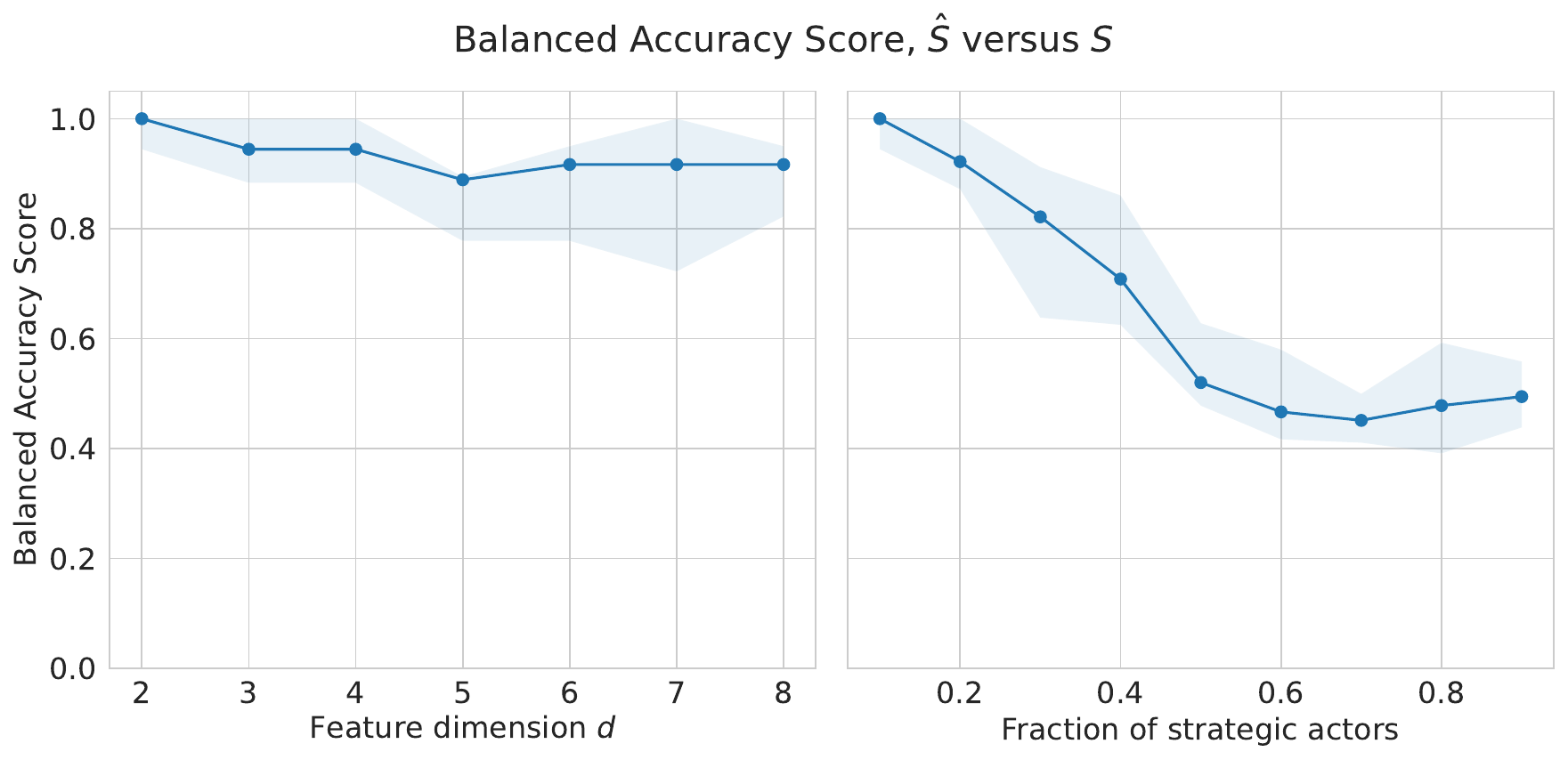}
    \caption{Balanced accuracy (the mean of the true positive rate and true negative rate) of $\hat S$ estimation, in the same settings as Figure~\ref{fig:regression_3wid}. Shaded regions denote $[10,90]$-percentile outcomes across $10$ independent trials.}
    \label{fig:acc_3wide}
\end{figure}

\subsection{Negotiations on International Trade Networks}

In this section, we simulate strategic negotiations on an international trade network among $n=46$ large economies~\cite{oecd-stats}.
Nodes represent nations, and edge $W_{ij}^t$ at time $t$ is the total recorded trade between $i$ and $j$ during a fiscal quarter.
Following~\cite{jalan-2023}, 
we infer the $(M^t, \Sigma^t, \Gamma^t)$ (Definition~\ref{def:negotiation}) from the networks $W^t$ over the period 2010-2020 (see Appendix for full details).

  

Trade networks arise from complex strategic considerations \cite{carlson2013game}. We model what would have happened if, in addition to their usual strategies, countries used Algorithm~\ref{alg:nash} in their trade negotiations. We compare the observed network against two situations: (a) only the United Kingdom (UK) is strategic, and (b) everyone is strategic. We check the utility of the UK and Netherlands under these situations (other countries have qualitatively similar results).

\begin{figure}[!htbp]
    \centering
    \includegraphics[width=0.6\textwidth]{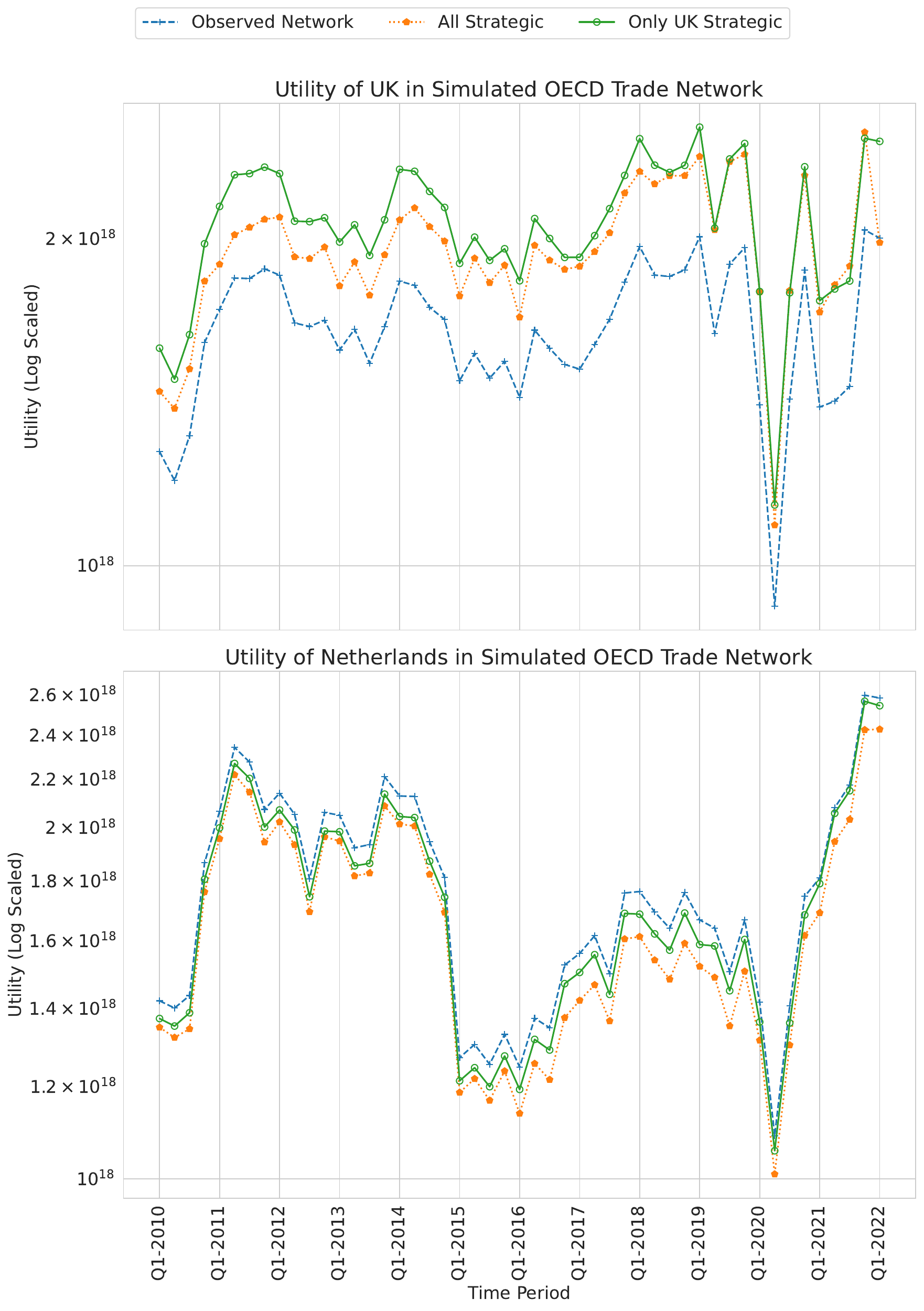}
    \caption{{\em Effect of strategic trading on the UK (top) and Netherlands (bottom):}
    The UK's utility is highest when it is the only strategic agent, and lowest when all are honest (the observed network). However, the pattern changes in the last quarter, where the utility when all are strategic is worse than when all are honest.
    The Netherlands has the highest utility when all others are honest, unlike the UK.
    }
    \label{fig:oecd}
\end{figure}

Figure~\ref{fig:oecd} shows the utility achieved by the UK and Netherlands for all quarters.
There is significant variability across quarters due to changes in $M_t$.
For both the UK and the Netherlands, their utility is reduced if other countries are strategic.
Interestingly, the UK can be {\em worse off} when all are strategic than when all are honest (the last quarter in Figure~\ref{fig:oecd}).
This validates our observations in Section~\ref{sec:welfare} that strategic behavior can sometimes lead to worse outcomes for the strategic actors when $\abs{S} > 1$. 




\section{Conclusions and Future Work}
\label{sec:conc}

In this paper, we propose a model of network formation with multiple strategic actors. Strategic agents manipulate the network formation process by using negotiating positions different from their true preferences. 
We give an efficient algorithm to find the set of all Nash equilibria, and show that they are all pure-strategy equilibria.
If there are multiple strategic agents, they may end up worse off at equilibrium than if they had all negotiated truthfully. This is in contrast to the single strategic agent setting where the strategic agent can do no worse than if she was honest. We also show that agents can learn the true preferences of others from historical network data, even if the others had negotiated strategically or sought to fool the learner. Finally, experimental results on real-world and simulated networks validate our approach.

Future work could include different noise models for learning, such as partially adaptive~\cite{mukhoty2023corruption} or oblivious noise~\cite{d2021consistent}. Further, one could extend our strategic negotiations model to a repeated games setting. To our knowledge, optimal negotiating positions in repeated games against learning agents have only been studied in non-network settings~\cite{deng-2019,assos2024maximizing}. Finally, in light of the results of Section~\ref{sec:welfare}, it would be interesting to determine the Price of Anarchy~\cite{roughgarden2005selfish} and other welfare measures in our model.

%% file: intro.tex

Many important processes involve pairwise interactions between agents.
For example, rumors and misinformation spread in a social network via interactions between pairs of friends~\cite{de2016learning,kleinberg-2020-opinion-dynamics,chen2021adversarial}.
Similarly, supply-chain networks of firms are useful in identifying bottlenecks and managing supply disruptions~\cite{elliott2022supply}.
Another example is the financial network, where each link represents a debt obligation between two firms.
A single firm defaulting on its obligations can bankrupt its creditors, setting off a cascade of defaulting firms~\cite{eiseinberg-noe-2001, feinstein2018sensitivity}.
Thus, the structure of the financial network affects the stability of the financial system.

Given the usefulness of networks, there has been increasing interest in games played between agents in a network \cite{leng-2020-learning, rossi2022learning, wang2024relationship,park2024multi}.
In such network games, the payoff of an agent depends on the actions of its neighbors in the network.
An example of this is the game of opinion dynamics.
Here, an agent (Alice) in a social network wants to influence her neighbors to adopt Alice's opinion on some topic.
To do so, Alice can express opinions that deviate from her actual opinion.
The optimal choice balances the cost of such deception against the utility gained by influencing her neighbors.

Prior works on such network games only consider a single strategic node or external actor
\cite{kleinberg-2020-opinion-dynamics, galeotti-2020,chen2021adversarial}.
Interactions between several strategic agents are not explored. 
Also, in these models, the network is pre-specified (exogenous).
Other authors have studied network formation games, but they largely assume that all agents report their actual preferences
~\cite{acemoglu-azar-2020, golub-sadler-2021, jalan-2023, wang2024relationship}.
Hence, they do not consider strategic interactions between agents.


We present, to our knowledge, the first results for a multi-agent network formation game with an arbitrary set of strategic agents.
In our model, heterogeneous agents negotiate to form a network of bilateral contracts. 
Each agent's utility depends on their private beliefs about the risks and rewards of their contracts. To get their desired contract sizes, agents try to deceive each other by presenting negotiating positions that can deviate from their actual beliefs.
All agents choose their negotiating positions privately.
Then, all negotiating positions are revealed simultaneously; agents cannot change their positions after viewing others' positions.
The contract sizes are then set via a formula that gives every agent the highest possible utility 
{\em assuming the negotiating positions reflect the agents' true beliefs.}
The resulting contract sizes determine the agents' actual utilities.
For this network formation game, we ask the following question:

\begin{center}
(Q1) What is the optimal negotiating position of each agent in a network?
\end{center}


At first sight, it is not obvious that an optimal position exists, especially when there are several strategic agents.
For instance, consider two hedge funds competing against each other to trade with an investor.
It may seem that the two funds would be locked in an arms race, leading to unbounded negotiating positions.
Furthermore, each fund must pick its position before seeing the other fund's choice.
This uncertainty makes the problem even more difficult.

Next, even for a single strategic agent, the optimal negotiating position requires some knowledge of other agents' preferences. This leads to our second question:
\begin{center}
(Q2) How can an agent learn the parameters needed for optimal negotiations?
\end{center}
Agents may observe the network from previous timesteps, but they cannot directly infer other agents' beliefs since the network was formed from strategic negotiations. Moreover, agents perceive correlations between their contracts and want to minimize their risk. Hence, agent $i$'s contract with $j$ can depend on $j$'s contract with $k$, which depends on $k$'s contract with $\ell$, and so on. So, an edge between $(i,j)$ can depend on the beliefs of all agents (including the strategic ones), not just $i$ and $j$.

\medskip 
\noindent Our main contributions are as follows. 

{\bf Efficient algorithm to find Nash equilibria.}
Our main result is an efficient Algorithm~\ref{alg:nash} to find the optimal negotiating positions for an arbitrary set of strategic agents, or report when no optimal solution exists (Theorem~\ref{thrm:indiv-optimal}).
We show that even in simple networks, the outcomes for strategic agents can vary considerably.
Depending on the correlations between agents and the set of strategic agents, the strategic agents might all be better off than being honest, or all worse off, or neither.

{\bf Learning algorithm for agents.} Given network data, we present an algorithm to learn the other agents' true beliefs and the set of strategic agents. 
Our algorithm is robust to strategic agents playing non-Nash-optimal strategies to fool the learner.

{\bf Experiments on simulated and real-world data.} 
We simulate Nash-optimal strategic negotiations on real-world international trade data~\cite{oecd-stats}. 
Our experiments confirm that the utilities of agents are sensitive to the set of strategic agents.
We also show that our learning algorithm recovers the parameters needed for strategic negotiations for a broad range of networks.

\section{Background and Related Work}\label{sec:background}

Our work is broadly concerned with strategic considerations in machine learning systems, and is at the intersection of growing fields such as artificial intelligence, algorithmic game theory, and multi-agent learning. 

{\bf Network games with strategic behavior and learning.} Network games, or graphical games, involve $n$ agents whose payoffs depend on the actions of their neighboring agents~\cite{kearns2001graphical, tardos-2004}. A large body of work studies learning from observations of network games~\cite{irfan-2014,garg-2016,de2016learning,leng-2020-learning,rossi2022learning}. All of these works study games with finite or one-dimensional action spaces, whereas each of the $n$ agents in our model has an $n$-dimensional action space. This results in a vector of contracts $\bm{w}_k \in \RR^n$ for each agent $k$. The correlations between entries of $\bm{w}_k$ can have strategic consequences, as we show in Section~\ref{sec:welfare}. 

A related line of work considers steering the outcome of a network game. These works largely consider games with one-dimensional action spaces and one strategic actor \cite{galeotti-2020,kleinberg-2020-opinion-dynamics,chen2021adversarial, wang2024relationship}.
Instead, we focus on consider arbitrary sets of strategic actors.
Some recent works study multiple strategic actors in repeated auctions \cite{kolumbus2022auctions} and Fisher markets with linear utilities \cite{kolumbus2024asynchronous}.
Our work has a similar thrust, but we focus on network formation.




{\bf Learning from strategic sources.} 
\cite{chen2020learning} study strategy-awareness for linear classification, but assume that agents can only misreport data up to an $\eps$-ball. Our setting is closer to that of \cite{ghalme2021strategic}, who show that agents who are evaluated by a third-party classifier (e.g. for approval for a bank loan) can strategically modify their features to game the classifier, even if the classifier used is strategy-robust in the sense of \cite{hardt2016strategic}. 
\cite{harris2023strategic} give a $\Tilde{O}(n^{(d+1)/(d+2)})$-regret algorithm for online binary classification against $n$ strategic agents with $d$-dimensional features, but in a linear reward model. Finally, \cite{daskalakis-2015} give a mechanism to encourage strategic data providers to report truthful data, but in a model where the data providers have no incentive to hurt the classifier's accuracy (e.g. crowdsourcing).

{\bf Multi-agent learning.} In terms of the ``five agendas'' of multi-agent learning \cite{shoham2007if}, we consider our work closest to the ``Prescriptive, non-cooperative'' agenda. The goal of this program is to design agents that are optimal for environments consisting of other learning agents. Several works study the problem of optimal play against a learning algorithm \cite{deng-2019, camara2020mechanisms, assos2024maximizing}, but each of these considers a single player against a single learner. Our work is more related to strategic behavior in network games, in which an agent must consider not only their neighbors, but the neighbors of their neighbors, and so on.

%% file: learning.tex
\section{Learning from Strategic Negotiation Outcomes}\label{sec:learning}



Algorithm~\ref{alg:nash} describes how a strategic agent should choose its Nash-optimal negotiating position.
However, to run the algorithm, the agent must know
\begin{enumerate}
 	\item the matrix $H:= (M + M^T)$ of the beliefs of all agents, and 
 	\item the set $S \subseteq [n]$ of strategic agents. 
\end{enumerate} 
Suppose an agent (called the {\em learner}) lacks this information but can observe the entire network $W^\prime$.
The learner wishes to learn $M$ and $S$ for use in future negotiations.
From $W^\prime$, the learner can recover the negotiating positions $M^\prime+{M^\prime}^T$ via Theorem~\ref{thrm:sylvester-eqns}.
But she cannot infer the true beliefs $M+M^T$.


We therefore consider a model in which the learner also knows extra information in the form of a feature matrix $X$. 

\begin{definition}[Network Setting with Agent Features]
Let $X \in \RR^{n \times d}$ have rows $\bx_i\in\RR^d$ for agent $i\in[n]$, and $B \in \RR^{d \times d}$. 
A network setting with agent features $(B, \Gamma, \Sigma, X)$ is such that, for $i, j \in [n]$
\[
M_{ij} = \bx_i^T B \bx_j
\]
for all $i, j \in [n]$. Hence, $\vc(M) = (X \otimes X) \vc(B)$.
\label{def:mmsb}
\end{definition}


For simplicity, we focus on the case of $d\ll n$ and $\Gamma=I$. The latter corresponds (up to a constant) to homogeneous risk aversions, which are commonly observed~\cite{ang-2014, paravisini-2017}.

Given $W^\prime$ and $X$, the learner wishes to learn $B$. Now, $W^\prime$ depends on the negotiating positions $M^\prime$, which can differ from $M$ in an unknown way. 
For instance, some strategic agents may occasionally deviate from the Nash strategy to throw off the learner.

 We therefore formulate our learning problem as a robust regression. 

\begin{definition}[Robust Regression]
Given features $X$ and a stable network $W^\prime$ arising from negotiating positions $M^\prime$, the robust regression problem with covariates $(X \otimes X)$, response $\by \in \RR^{n^2}$, and corruption threshold $\beta \in [0,1]$ is to solve: 
\[
\min\limits_{\hat B \in \RR^{d \times d}} \norm (X \otimes X) \vc(\hat B + \hat B^T) - \bm{y} \norm_2^2,
\]
assuming that $\bm{y}$ differs from $\vc(M + M^T)$ arbitrarily at up to $\beta n^2$ entries. We write \textsc{RR}$(X \otimes X, \bm{y}, \beta)$ for shorthand. 
\end{definition}
Note that we do not require that $M^\prime$ corresponds to a Nash equilibrium. 
We assume an adaptive adversary as opposed to the simpler oblivious adversary setting.
Hence, the learner must learn in the face of complex counter-strategies by other agents who have full knowledge of the network setting and features.
\begin{algorithm}[t]
	\KwIn{Network $W^\prime \in \RR^{n \times n}$, agent features $X \in \RR^{n \times d}$, corruption threshold $\beta \in (0,1)$}
	\KwOut{Estimates $\hat B \in \RR^{d \times d}$ and $\hat S \subseteq [n]$.}
    $\vc(H^\prime) \gets K\vc(W^\prime)$\\
	$\vc(\hat B) \gets$ solve \textsc{RR}$(X \otimes X, \vc(H^\prime), \beta)$\\ 
    $R_{ij} \gets |\bme_i^T (H^\prime - \frac 1 2 X (\hat B + \hat{B}^T) X^T) \bme_j|$\\
	$A_{ij} = \begin{cases}
	1 & \text{ if $R_{ij}$ belongs to top $\beta n^2$ entries of $R$} \\
	0 & \text{otherwise}
	\end{cases}
    $\\
	$S_1, S_2 \gets$ spectral clustering on $A$ with $2$ clusters\\
	$\hat S \gets S_i$ such that $\abs{S_i}$ is nearest to $\frac{\sqrt{8 \beta n} + 1}{4}$\\ 
	\Return{$\hat{B}, \hat S$.}
	\caption{Estimation of mean beliefs $B$ and strategic agents $S$}
\label{alg:learning}
\end{algorithm}

We use \textsc{Torrent}~\cite{torrent-2015} to solve the robust regression problem. This algorithm has provable guarantees even when an $\Omega(1)$ fraction of the response vector $\bm{y}$ is adversarially corrupted. However, we emphasize that the choice of robust regression algorithm is independent of our method, and alternative algorithms can also be used.

Our Algorithm~\ref{alg:learning} first learns a $\hat B$ using \textsc{Torrent}. Then, it computes a matrix of residuals $R$ and constructs an unweighted graph $G = ([n], E)$ such that $(i, j) \in E$ iff $R_{ij}$ is one of the $\beta n^2$ largest residuals. If $\hat B \approx B$, then these edges should all be incident to the strategic nodes $S \subset [n]$. Therefore, a consistent clustering algorithm such as~\cite{rohe2011spectral} can recover $S$.

Next, we discuss the recovery guarantee for $\hat{B}$.
We defer the theoretical guarantee for $\hat{S}$ to the Appendix. 
We recount the following technical condition of \cite{torrent-2015}. For a matrix $X \in \RR^{n \times d}$ with $n$ samples in $\RR^d$ and $S \subset [n]$ let $X_S \in \RR^{\abs{S} \times d}$ select rows in $S$. 
\begin{definition}[SSC and SSS Conditions]
Let $\gamma \in (0,1)$. A design matrix $X \in \RR^{n \times d}$ satisfies the Subset Strong Convexity Property at level $1 - \gamma$ and Subset Strong Smoothness Property at level $\gamma$ with constants $\lambda_{1-\gamma}, \Lambda_\gamma$ respectively if: 
\begin{align*}
\lambda_{1-\gamma} &\leq 
\min\limits_{S \subset [n]: \abs{S} = (1-\gamma) n} \lambda_{min}(X_S^T X_S) \\
\Lambda_{\gamma} &\geq \max\limits_{S \subset [n]: \abs{S} = \gamma n} \lambda_{max}(X_S^T X_S)
\end{align*}
\label{defn:ssc}
\end{definition}
We will give a concrete example of the SSC and SSS constants for a feature matrix $X$ in Example~\ref{example:sbm}. 

\begin{example}[Balanced Stochastic Block Model]
Suppose that $X \in \{0,1\}^{n \times 2}$ describes community memberships for a Stochastic Block Model with equally sized communities. Then $(X \otimes X)^T (X \otimes X) = \frac{n^2}{4} I_4$. It can be shown that for any $\gamma \in (0,1)$ that the corresponding constants are $\lambda_{1-\gamma} = n^2(\frac 1 4 - \gamma)$ and $\Lambda_\gamma = \gamma n^2$. Therefore the condition of Proposition~\ref{prop:b-recovery} holds for $\gamma < \frac{1}{68}$. A sufficient condition is $\abs{S} \leq \frac{n}{136}$. 
\label{example:sbm}
\end{example}




\begin{proposition}
Suppose $S \subset [n]$ is the strategic set, and $M^\prime \in \RR^{n \times n}$ is the matrix of negotiating positions that results in a stable network $W^\prime$. 
Let $\beta \geq \frac{2n\abs{S} - \abs{S}^2}{n^2}$. Suppose $X \otimes X$ satisfies the SSC condition at level $1 - \beta$ with constant $\lambda_{1-\beta}$, and SSS condition at level $\beta$ with constant $\Lambda_\beta$ (Definition~\ref{defn:ssc}). 
Then, there exists a constant $C > 0$ such that if $\abs{S} \leq Cn$ and $4\frac{\sqrt{\Lambda_{\alpha}}}{\sqrt{\lambda_{1-\alpha}}} < 1$, Algorithm~\ref{alg:learning} with threshold parameter $\beta$ and $T$ iterations of \textsc{Torrent} returns $\hat B$ such that: 
\begin{align*}
\norm \hat B + \hat B^T - (B + B^T) \norm_F 
&\leq \exp(-c T) \\
&\cdot \bigg(
\frac{\norm M^\prime + (M^\prime)^T - (M + M^T) \norm_F}{n} \bigg)
\end{align*}
for an absolute constant $c > 0$, and $n$ large enough. 
\label{prop:b-recovery}
\end{proposition}

\begin{remark}[Random design]
The SSC and SSS conditions are known to hold for a sub-Gaussian design \cite{torrent-2015}. Our Proposition~\ref{prop:b-recovery} concerns arbitrary fixed design, but can be easily extended to the random design setting with similar techniques. 
\end{remark}



%% file: appendix.tex


\section{Proofs and Additional Theoretical Results}

In this section we give proofs for the results of the paper, and also include additional theoretical results (Section~\ref{sec:stochasticM} and Section~\ref{sec:s-recovery}) discussed in the paper. Note that the proof of Proposition~\ref{prop:3p-insights} will be in Section~\ref{sec:3p-no-mechanism}, along with the discussion of an additional model network.

\subsection{Proof of Theorem \ref{thrm:lying-bounded}}

\begin{proposition}
Let $k \in [n]$. Let $\bdelta_{-k} := \{\bdelta_i: i \neq k\}$ and let $(W, P)$ be the stable point if $k$ reports honestly and all others report according to $\bdelta_{-k}$. 
Next, consider some $\bdelta_k \neq \bm{0}$ and let $(W^\prime, P^\prime)$ be the stable point if $k$ reports 
$\bdelta_k$ and all others report according to $\bdelta_{-k}$. 
Then $W^\prime \be_k = W \be_k + B\bdelta_k$ for a matrix $B$ defined as follows. Let $\Gamma^{-1/2} \Sigma \Gamma^{-1/2} = V \Lambda V^T$ be the eigendecomposition of $\Gamma^{-1/2} \Sigma \Gamma^{-1/2} \succ 0$. Let $A \in \RR^{n \times n}$
 be a symmetric matrix such that: 
\[
A_{ij} = \begin{cases}
\frac{V_{ki}^2}{4\lambda_i} + \sum\limits_{\ell=1}^{n} \frac{V_{k\ell}^2}{2(\lambda_i + \lambda_\ell)} & i = j \\
\frac{V_{ki}V_{kj}}{2(\lambda_i + \lambda_j)} & i \neq j
\end{cases} 
\]
Then $B = \gamma_k^{-1} \Gamma^{-1/2} V AV^T \Gamma^{-1/2}$. 
\label{prop:contract-shift}

\end{proposition}

\begin{proof}
Let $\Delta_W = W^\prime - W$. By Theorem \ref{thrm:sylvester-eqns}, $2(\Sigma \Delta_W \Gamma + \Gamma \Delta_W \Sigma) = \bme_k \bdelta_k^T + \bdelta_k \bme_k^T$. Therefore $\vc(\Delta_W) = \frac 1 2 (\Sigma \otimes \Gamma + \Gamma \otimes \Sigma)^{-1} (\bme_k \bdelta_k^T + \bdelta_k \bme_k^T)$. 

Next, let $\bm{v}_i := V \bm{e}_i$. 
Using the eigendecomposition properties of Kronecker sums \cite{horn-johnson-topics-2008} as in Corollary~1 of~\cite{jalan-2023},
\begin{align*}
\Gamma^{1/2} \Delta_W \Gamma^{1/2} &= \bigg(
\sum\limits_{i = 1}^{n} \sum\limits_{j=1}^{n} \frac{\bmv_i^T \Gamma^{-1/2}
\big(\bme_k \bdelta_k^T
\big)
\Gamma^{-1/2}\bmv_j }{2(\lambda_i + \lambda_j)} \bmv_i \bmv_j^T
 \\
&+ 
\sum\limits_{i = 1}^{n} \sum\limits_{j=1}^{n} \frac{\bmv_i^T \Gamma^{-1/2}
\big(\bdelta_k \bme_k^T
\big)
\Gamma^{-1/2}\bmv_j }{2(\lambda_i + \lambda_j)} \bmv_i \bmv_j^T
\bigg)
\end{align*}
Let $G := \Gamma^{-1/2} V$ and $\bm{g}_i := G \be_i$. Then $\bmv_i^T \Gamma^{-1/2} \be_k = G_{ki}$ and $\be_k^T \Gamma^{-1/2}\bmv_j = G_{kj}$. Hence:
\begin{align*}
\Delta_W \be_k &= 
\Gamma^{-1/2}
\bigg(
\sum\limits_{i = 1}^{n} \sum\limits_{j=1}^{n} \frac{G_{ki} \bm{g}_j^T \bdelta_k + G_{kj} \bm{g}_i^T \bdelta_k}{2(\lambda_i + \lambda_j)} \bmv_i \bmv_j^T
\bigg)\Gamma^{-1/2} \be_k \\
&= \Gamma^{-1/2}
\bigg(
\sum\limits_{i = 1}^{n} \sum\limits_{j=1}^{n} \frac{G_{ki} \bm{g}_j^T \bdelta_k + G_{kj} \bm{g}_i^T \bdelta_k}{2(\lambda_i + \lambda_j)} G_{kj} \bmv_i 
\bigg) \\
&= \Gamma^{-1/2}
\bigg(
\sum\limits_{i = 1}^{n} \sum\limits_{j=1}^{n} \frac{G_{ki}G_{kj}}{2 (\lambda_i + \lambda_j)} \bmv_i \bmv_j^T \\
&+ \sum\limits_{j=1}^{n} G_{kj}^2 \cdot \bigg(
\sum\limits_{i=1}^{n} \bigg(\frac{1}{2 (\lambda_i + \lambda_j)} \bmv_i \bmv_i^T\bigg)
\bigg) \bigg) \Gamma^{-1/2} \bdelta_k
\end{align*}
Hence $W^\prime \bme_k - W\bme_k= B \bdelta_k$ for a matrix $B$ defined as above. We can further simplify $B$ as 
\begin{align*}
B &= \Gamma^{-1/2} \bigg(
\sum\limits_{i = 1}^{n} 
\bigg(
\frac{G_{ki}^2}{4\lambda_i} + \sum\limits_{j=1}^{n} \frac{G_{kj}^2}{2(\lambda_i + \lambda_j)}
\bigg)
\bmv_i \bmv_i^T \\
&+ \sum\limits_{i=1}^{n} \sum\limits_{j \neq i} \frac{G_{ki}G_{kj}}{2(\lambda_i + \lambda_j)} \bmv_i \bmv_j^T
\bigg) \Gamma^{-1/2}
\end{align*}
Notice that $G_{ki} = \gamma_k^{-1/2} V_{ki}$, and that $B$ depends only on the $k^{th}$ row of $G$ so we can factor out $\gamma_k^{-1}$ and replace the entries $G_{ki}$ with $V_{ki}$. 

Finally, let $A \in \RR^{n \times n}$ be defined as in the statement of this Proposition. Then we conclude that 
\begin{align*}
B &= \gamma_k^{-1} \Gamma^{-1/2} \big(\sum\limits_{i=1}^{n} A_{ii} \bmv_i \bmv_i^T + \sum\limits_{i=1}^{n} \sum\limits_{j\neq i} A_{ij} \bmv_i \bmv_j^T \big) \Gamma^{-1/2} \\
&= \gamma_k^{-1} \Gamma^{-1/2} V A V^T \Gamma^{-1/2}	
\end{align*}
\end{proof}

\begin{proposition}
Let $K = \Lambda^{1/2} A \Lambda^{1/2}$. The eigenvalues of $K$ are all real and contained in $(0, 1)$. 
\label{prop:k-eigenvalues}
\end{proposition}

\begin{proof}
Notice that $K$ is symmetric, since $\Lambda$ is diagonal and $A$ is symmetric. By the Spectral Theorem, $K$ has an eigendecompsition with real eigenvalues. 

Next, notice $A = C + D$ for $C_{ij} = C_{ji} = \frac{V_{ki}V_{kj}}{2(\lambda_i + \lambda_j)}$ and $D$ a diagonal matrix with $D_{ii} = \sum\limits_{\ell=1}^{n} \frac{V_{k\ell}^2}{2(\lambda_i + \lambda_\ell)}$. Then $K = \t C + \t D$ for $\t C = \Lambda^{1/2} C \Lambda^{1/2}$ and $\t D = \Lambda^{1/2} D \Lambda^{1/2}$.

Let $\bm{x} = \frac{1}{\sqrt{2}} \Lambda^{1/2} V^T \bm{e}_k$. 
Then $\t C$ satisfies the Lyapunov equation:
\[
\Lambda \t C + \t C \Lambda = \bm{x}\bm{x}^T
\]
Since $\t C$ is self-adjoint it has an eigenbasis with real eigenvalues. Let $\bm{y}$ be an eigenvector of $\t C$ with eigenvalue $\mu$. Then $(\bm{y}^T \bm{x})^2 = \bm{y}^T (\t C \Lambda + \Lambda \t C)\bm{y} = 2\mu \bm{y}^T \Lambda \bm{y}$. By the Cauchy-Schwarz inequality, 
\begin{align*}
\mu &= \frac{(\bm{y}^T \bm{x})^2}{2 \bm{y}^T \Lambda \bm{y}} \\
&= \frac{1}{4} \frac{(\sum\limits_{i=1}^{n} \sqrt{\lambda_i} y_i V_{ki})^2}{\sum\limits_{i=1}^{n} \lambda_i y_i^2} \\
&\leq \frac{1}{4} \frac{(\sum\limits_{i=1}^{n} \lambda_i y_i^2) 
\sum\limits_{i=1}^{n} V_{ki}^2)}{\sum\limits_{i=1}^{n} \lambda_i y_i^2} \\
&\leq \frac{1}{4}
\end{align*}
Further, since $\Lambda \succ 0$, $\mu \geq 0$. So the eigenvalues of $\t C$ are all within $[0, \frac 1 4]$. 
Next, the eigenvalues of $\t D$ are simply its diagonal entries. 
Recall that
\begin{align*}
\t D_{ii} &= \sum\limits_{j} \frac{\lambda_i V_{kj}^2}{2(\lambda_i + \lambda_j)} > \frac{\lambda_i}{2(\lambda_i + \max_j\lambda_j)} \sum_j V_{kj}^2 > 0,
\end{align*}
since $V$ is orthonormal and $\lambda_j>0$.
By similar reasoning, $\t D_{ii} < \frac 1 2$. We conclude that the eigenvalues of $K$ are contained in $(0, \frac 3 4)$. 
\end{proof}

We are ready to prove Theorem \ref{thrm:lying-bounded}.
\begin{proof}[Proof of Theorem \ref{thrm:lying-bounded}.]
Fix the index $k$ of the strategic actor. Suppose $k$ reports $\bdelta_k$ and each $i \neq k$ reports some $\bdelta_i$. Let $(W^\prime, P^\prime)$ be the resulting stable point. 

By Lemma \ref{lemma:liar-util}, the utility of $k$ at $(W^\prime, P^\prime)$ is $-\la \bdelta_k, \bw_k^\prime \ra + \gamma_k \la \bw_k^\prime, \Sigma \bw_k^\prime \ra$ where $\bw_k^\prime$ is the $k^{th}$ column of $W^\prime$. 

By Proposition \ref{prop:contract-shift}, $\bw_k^\prime = \bw_k + B\bdelta_k$ for $B = \gamma_k^{-1} \Gamma^{-1/2} V AV^T \Gamma^{-1/2}$ with $V$ and $A$ as in Proposition \ref{prop:contract-shift}. Hence the utility of $k$ is quadratic in $\bdelta_k$, and the quadratic term is $-\bdelta_k^T \big(B - \gamma_k B \Sigma B\big)\bdelta_k$. A straightforward calculation gives: 
\begin{align*}
B - \gamma_k B \Sigma B &= \gamma_k^{-1} \Gamma^{-1/2} V\Lambda^{-1/2}
\bigg(
\Lambda^{1/2} A \Lambda^{1/2} \\
&- \big(\Lambda^{1/2} A \Lambda^{1/2} \big)^2
\bigg)
\Lambda^{-1/2} V^T \Gamma^{-1/2}
\end{align*}
Let $K := \Lambda^{1/2} A \Lambda^{1/2}$. To show that the Hessian of the utility of $k$ is negative definite in $\bdelta_k$, we need to show $K - K^2 \succ 0$. Since the spectrum of $K$ is contained in $(0, 1)$ by Proposition \ref{prop:k-eigenvalues}, $K \succ K^2$ and the conclusion follows. 
\end{proof}


\subsection{Proof of Theorem~\ref{thrm:indiv-optimal}}

We require the following Lemmata.

\begin{lemma}[\cite{seber2012linear}]
Let $\bm{x} \in \RR^n$ be a random vector and $A \in \RR^{n \times n}$ a symmetric matrix. Then, if $\EE[\bm{x}]=\bmu$ and $\bm{x}$ has covariance $\Sigma$, then: 
\[
\EE[\bm{x}^T A \bm{x}] = tr(A \Sigma) + \bmu^T A \bmu
\]
\label{lemma:expected-quadratic}
\end{lemma}




A straightforward calculation gives the following.
\begin{lemma}[Utility from Strategy]
\label{lemma:liar-util}
If agent $k$ reports $\bdelta_k^\prime$ resulting in $(W^\prime, P^\prime)$, then $k$'s utility is 
\begin{align*}
g_k(W^\prime, P^\prime) &= -\la \bdelta_k, \bw_k^\prime \ra
+ \gamma_k \la \Sigma \bw_k^\prime, \bw_k^\prime \ra, 
\end{align*}
where $\bw_k^\prime = W^\prime \be_k$.
\end{lemma}

Next, recall the commutation matrix.
\begin{lemma}[\cite{horn-johnson-topics-2008}]
There exists a permutation matrix $\Pi: \RR^{n^2} \to \RR^{n^2}$ such that for $X \in \RR^{n \times n}$, $\Pi \vc(X) = \vc(X^T)$. We call $\Pi$ the commutation matrix. 
\label{lemma:commutation-matrix}
\end{lemma}

We are ready to prove Theorem~\ref{thrm:indiv-optimal}. 

\begin{proof}[Proof of Theorem~\ref{thrm:indiv-optimal}.]
We begin by proving that agent $k$'s optimal negotiating position $\bdelta_k^\star$ given $\bdelta_{j \in S: j \neq k}^*$ is the solution to the linear system: 
\begin{align}
(2 \gamma_k (L^{(k, k)})^T \Sigma - I)\bm{v}_k = T^{(k,k)} \bdelta_k^\star + \sum\limits_{j \in S: j \neq k} T^{(k,j)} \bdelta_j^\star
\label{eq:toprove1}
\end{align}
where $T^{(p,q)}$ and $L$ are defined as in Algorithm~\ref{alg:nash}.

Let $\Delta_M \in \RR^{n \times n}$ have $i^{th}$ column $\bdelta_i^\star$ if $i \in S$ and zero otherwise.
Let $(W^\prime, P^\prime)$ be the stable point resulting from a choice of $M^\prime:= M + \Delta_M$ as the agents' negotiating positions in the Strategy Phase.
From Theorem \ref{thrm:sylvester-eqns}, we have
\begin{align*}
\vc(W^\prime) &= \vc(W) \\
&+ 0.5(\Sigma \otimes \Gamma + \Gamma \otimes \Sigma)^{-1} \vc(\Delta_M + \Delta_M^T)\\
\Rightarrow \vc(W^\prime - W) &= L \vc(\Delta_M) \\
\Rightarrow \bmw_k^\prime - \bmw_k &= L^{(k,k)} \bdelta_k + \sum\limits_{j \in S: j \neq k} L^{(k, j)} \bdelta_j,
\end{align*}
where the second line follows from Lemma~\ref{lemma:commutation-matrix}, and the matrix $L$ is defined as in Algorithm~\ref{alg:nash}.
Notice that there is a distribution on $\bmw_k$ induced by $\mathcal{D}_k$.

Now, fix an agent $k \in S$. They want to choose $\bm{w}_k^\prime$ optimally based on the above equation, but are uncertain about the value of $\bmw_k$. 

Let $A_k := L^{(k, k)}$ and $\bm{b}_k := \sum\limits_{j \in S: j \neq k} L^{(k, j)} \bdelta_j$.
By Lemma \ref{lemma:liar-util}, we have:
\begin{align*}
g_k &= -\la \bdelta_k, \bmw_k^\prime \ra 
+ \gamma_k \la \bmw_k^\prime, \Sigma \bmw_k^\prime \ra \\
&= -\la \bdelta_k, (\bmw_k + A_k \bdelta_k + \bm{b}_k) \ra \\
&  +\gamma_k \la (\bmw_k + A_k \bdelta_k + \bm{b}_k), \Sigma (\bmw_k + A_k \bdelta_k + \bm{b}_k) \ra \\
&= -\la \bdelta_k, A_k \bdelta_k \ra - \la \bdelta_k, \bm{b}_k \ra - \la \bdelta_k, \bmw_k \ra 
+ 2 \gamma_k  \la \Sigma \bmw_k, A_k \bdelta_k + \bm{b}_k \ra
\\
& + \gamma_k  \la A_k \bdelta_k + \bm{b}_k, \Sigma (A_k \bdelta_k + \bm{b}_k) \ra 
+ \gamma_k  \la \bmw_k, \Sigma \bmw_k \ra
\end{align*}

Agent $k$ wants to optimize $\EE_{\bm{b}_k}[g_k]$ by choosing $\bdelta_k$. Notice $\bm{b}_k$ is a linear function of the vectors $\bdelta_j$. Let $\bmu = \EE[\bm{b}_k]$ and $Q = \EE[(\bm{b}_k - \bmu)(\bm{b}_k - \bmu)^T]$. 

Therefore agent $k \in S$ wants to optimize: 
\begin{align*}
\EE_{\bm{b}_k}[g_k] &=
\EE_{\bm{b}_k}
\bigg[
-\la \bdelta_k, A_k \bdelta_k \ra - \la \bdelta_k, \bm{b}_k \ra - \la \bdelta_k, \bmw_k \ra \\
&+ 2 \gamma_k  \la \Sigma \bmw_k, A_k \bdelta_k + \bm{b}_k \ra
\\
& + \gamma_k  \la A_k \bdelta_k + \bm{b}_k, \Sigma (A_k \bdelta_k + \bm{b}_k) \ra \\
&+ \gamma_k  \la \bmw_k, \Sigma \bmw_k \ra
\bigg] \\
&= - \la \bdelta_k, A_k \bdelta_k \ra 
- \la \bdelta_k, \bmu \ra 
- \la \bdelta_k, \bw_k \ra \\
&+ 2 \gamma_k \la \Sigma \bw_k, A_k \bdelta_k \ra 
+ 2 \gamma_k \la \Sigma \bw_k, \bmu \ra 
+ \gamma_k \la \bw_k, \Sigma \bw_k \ra \\
&+ \gamma_k \la \Sigma A_k \bdelta_k, A_k \bdelta_k \ra 
+ 2 \gamma_k \la \bmu, \Sigma A_k \bdelta_k \ra \\
&+ \gamma_k \tr(\Sigma Q)
+ \gamma_k \la \bmu, \bmu \ra
\end{align*}
Where the last step is by Lemma~\ref{lemma:expected-quadratic}. 

Next, by Theorem \ref{thrm:lying-bounded}, the optimal negotiating position $\bdelta^\star$ is the critical point of $g_k$ with respect to $\bdelta_k$. Notice that the Hessian of $\EE[g_k]$ with respect to $\bdelta_k$ does not depend on $\bdelta_j$ for any $j \neq k$, so the critical point gives the optimal negotiating position for $\EE[g_k]$. Setting $\nabla_{\delta_{k}}\EE_{\bm{b}_k}[g_k] = 0$, we obtain: 

\begin{align*}
(A_k + A_k^T - 2 \gamma_k A_k^T \Sigma A_k) \bdelta_k^\star &= (2 \gamma_kA_k^T \Sigma - I) \bm{w}_k \\
&+ (2 \gamma_k A_k^T \Sigma - I) \bmu \nonumber
\end{align*}
Notice that the gradients of the quadratic terms $\tr(\Sigma R_k)$ and $\gamma_k^T \bv_k^T \Sigma \bv_k$ with respect to $\bdelta_k$ are zero. 

Rearranging terms, we obtain the linear system: 
\begin{align}
(2 \gamma_k (L^{(k, k)})^T \Sigma - I)\bm{v}_k \nonumber &=
\bigg[(L^{(k, k)} + (L^{(k, k)})^T \\
&- 2 \gamma_k (L^{(k, k)})^T \Sigma L^{(k, k)}) \bdelta_k^\star \nonumber \\
&+ \sum\limits_{j \in S: j \neq k} L^{(k, j)} 
\EE[\bdelta_j] \\
&- 2 \gamma_k (L^{(k, k)})^T \Sigma \sum\limits_{j \in S: j \neq k} L^{(k, j)} 
\EE[\bdelta_j]
\bigg]
\nonumber \\
\Rightarrow \bm{z}^{(k)} &= T^{(k,k)} \bdelta_k^\star + \sum\limits_{j \in S: j \neq k} T^{(k,j)} \EE[\bdelta_j]
\label{eq:nash-linear-modified}
\end{align}
Where $T^{(k,k)} = (L^{(k, k)} + (L^{(k, k)})^T - 2 \gamma_k (L^{(k, k)})^T \Sigma L^{(k, k)})$ and $T^{(k,j)} = (I - 2 \gamma_k (L^{(k,k)})^T \Sigma) L^{(k, j)}$ for $j \neq k$. 


\end{proof}

\subsection{Generalization to Stochastic M}\label{sec:stochasticM}


In this section we will describe optimal negotiating positions in the setting where each strategic agent does not know the true matrix $M \in \RR^{n \times n}$ but instead has a probability distribution for it. The proof is similar to that of Theorem~\ref{thrm:indiv-optimal}. 


\newcommand{\D}{\mathcal{D}}

\begin{theorem}

Suppose agent $i\in[n]$ believes $M \in \RR^{n \times n}$ follows $M \sim \mathcal{D}_i$, and seeks to maximize its expected utility $\EE_{\D_i}[g_i]$.
We assume that all distributions $\mathcal{D}_i$ have finite first and second moments.
Let $V_i \in \RR^{n \times n}$ be such that $\vc(V_i) = \frac{1}{2}K^{-1}(\EE_{\D_i}[M] + \EE_{\D_i}[M]^T)$, where $K$ is defined in Algorithm~\ref{alg:nash}.
Let $\bm{v}_i = V_i \be_i$. 
Suppose each strategic agent $k\in S$ knows $\{\bm{v}_j\mid j\in S\}$ (they can compute it from the network setting $(\{\mathcal{D}_i\}, \Gamma, \Sigma))$ which is known to all strategic agents).
Modify the linear system of Algorithm~\ref{alg:nash} so that: 
\[
\forall k \in S: \by^{(k)} \leftarrow (2 \gamma_k (L^{(k, k)})^T \Sigma - I)\bm{v}_k
\]
This modified version of Algorithm~\ref{alg:nash} returns the set of Nash equilibria if they exist, and otherwise returns ``No Nash Equilibrium.''


\end{theorem}
\begin{proof}
We begin by proving that agent $k$'s optimal negotiating position $\bdelta_k^\star$ given $\bdelta_{j \in S: j \neq k}^*$ is the solution to the linear system: 
\begin{align}
(2 \gamma_k (L^{(k, k)})^T \Sigma - I)\bm{v}_k = T^{(k,k)} \bdelta_k^\star + \sum\limits_{j \in S: j \neq k} T^{(k,j)} \bdelta_j^\star
\label{eq:toprove1}
\end{align}
where $T^{(p,q)}$ and $L$ are defined as in Algorithm~\ref{alg:nash}.

Let $\Delta_M \in \RR^{n \times n}$ have $i^{th}$ column $\bdelta_i^\star$ if $i \in S$ and zero otherwise.
Let $(W^\prime, P^\prime)$ be the stable point resulting from a choice of $M^\prime:= M + \Delta_M$ as the agents' negotiating positions in the Strategy Phase.
From Theorem \ref{thrm:sylvester-eqns}, we have
\begin{align*}
\vc(W^\prime) &= \vc(W) \\
&+ 0.5(\Sigma \otimes \Gamma + \Gamma \otimes \Sigma)^{-1} \vc(\Delta_M + \Delta_M^T)\\
\Rightarrow \vc(W^\prime - W) &= L \vc(\Delta_M) \\
\Rightarrow \bmw_k^\prime - \bmw_k &= L^{(k,k)} \bdelta_k + \sum\limits_{j \in S: j \neq k} L^{(k, j)} \bdelta_j,
\end{align*}
where the second line follows from Lemma~\ref{lemma:commutation-matrix}, and the matrix $L$ is defined as in Algorithm~\ref{alg:nash}.
Notice that there is a distribution on $\bmw_k$ induced by $\mathcal{D}_k$.

Now, fix an agent $k \in S$. They want to choose $\bm{w}_k^\prime$ optimally based on the above equation, but are uncertain about the value of $\bmw_k$. 

Let $A_k := L^{(k, k)}$ and $\bm{b}_k := \sum\limits_{j \in S: j \neq k} L^{(k, j)} \bdelta_j^*$.
By Lemma \ref{lemma:liar-util}, we have:
\begin{align*}
g_k &= -\la \bdelta_k, \bmw_k^\prime \ra 
+ \gamma_k \la \bmw_k^\prime, \Sigma \bmw_k^\prime \ra \\
&= -\la \bdelta_k, (\bmw_k + A_k \bdelta_k + \bm{b}_k) \ra \\
&  +\gamma_k \la (\bmw_k + A_k \bdelta_k + \bm{b}_k), \Sigma (\bmw_k + A_k \bdelta_k + \bm{b}_k) \ra \\
&= -\la \bdelta_k, A_k \bdelta_k \ra - \la \bdelta_k, \bm{b}_k \ra - \la \bdelta_k, \bmw_k \ra 
+ 2 \gamma_k  \la \Sigma \bmw_k, A_k \bdelta_k + \bm{b}_k \ra
\\
& + \gamma_k  \la A_k \bdelta_k + \bm{b}_k, \Sigma (A_k \bdelta_k + \bm{b}_k) \ra 
+ \gamma_k  \la \bmw_k, \Sigma \bmw_k \ra
\end{align*}
Therefore agent $k \in S$ wants to optimize: 
\begin{align*}
\EE_{M \sim \mathcal{D}_k}[g_k] &=
\EE_{M \sim \mathcal{D}_k}
\bigg[
-\la \bdelta_k, A_k \bdelta_k \ra - \la \bdelta_k, \bm{b}_k \ra - \la \bdelta_k, \bmw_k \ra \\
&+ 2 \gamma_k  \la \Sigma \bmw_k, A_k \bdelta_k + \bm{b}_k \ra
\\
& + \gamma_k  \la A_k \bdelta_k + \bm{b}_k, \Sigma (A_k \bdelta_k + \bm{b}_k) \ra \\
&+ \gamma_k  \la \bmw_k, \Sigma \bmw_k \ra
\bigg] \\
&= -\la \bdelta_k, A_k \bdelta_k \ra 
- \la \bdelta_k, \bm{b}_k \ra \\
&+ \gamma_k  \la A_k \bdelta_k + \bm{b}_k, \Sigma(A_k \bdelta_k + \bm{b}_k)\ra
\\
&+ \EE_{M \sim \mathcal{D}_k}
\bigg[- \la \bdelta_k, \bmw_k \ra 
+ 2 \gamma_k  \la \Sigma \bmw_k, A_k \bdelta_k + \bm{b}_k \ra \\
&+ \gamma_k  \la \bmw_k, \Sigma \bmw_k \ra
\bigg]
\end{align*}
Next, recall that $\EE_{M \sim \mathcal{D}_k}[\bw_k] = \bv_k$. Let $\bw_k$ have covariance $R_k$. Then, by Lemma~\ref{lemma:expected-quadratic}, we have: 
\begin{align*}
\EE_{M \sim \mathcal{D}}[g_k] 
&= -\la \bdelta_k, A_k \bdelta_k \ra 
- \la \bdelta_k, \bm{b}_k \ra \\
& + \gamma_k  \la A_k \bdelta_k + \bm{b}_k, \Sigma(A_k \bdelta_k + \bm{b}_k)\ra \\
&+ \bigg(
- \bdelta_k^T \bv_k 
+ 2 \gamma_k \bv_k^T \Sigma A_k \bdelta_k \\
&+ 2 \gamma_k \bv_k^T \Sigma \bm{b}_k
+ \gamma_k tr(\Sigma R_k)
+ \gamma_k \bv_k^T \Sigma \bv_k
\bigg)
\end{align*}


By Theorem \ref{thrm:lying-bounded}, the optimal negotiating position $\bdelta^\star$ is the critical point of $g_k$ with respect to $\bdelta_k$. Notice that the Hessian of $\EE[g_k]$ with respect to $\bdelta_k$ does not depend on $\mathcal{D}$, so the critical point gives the optimal negotiating position for $\EE[g_k]$. Setting $\nabla_{\delta_{k}}\EE_{M \sim \mathcal{D}}[g_k] = 0$, we obtain: 
\begin{align*}
(A_k + A_k^T - 2 \gamma_k A_k^T \Sigma A_k) \bdelta_k^\star &= (2 \gamma_kA_k^T \Sigma - I) \bm{v}_k \\
&+ (2 \gamma_k A_k^T \Sigma - I) \bm{b}_k \nonumber
\end{align*}
Notice that the gradients of the quadratic terms $\tr(\Sigma R_k)$ and $\gamma_k^T \bv_k^T \Sigma \bv_k$ with respect to $\bdelta_k$ are zero. 

Rearranging terms, we obtain the linear system: 
\begin{align}
(2 \gamma_k (L^{(k, k)})^T \Sigma - I)\bm{v}_k \nonumber &=
\bigg[(L^{(k, k)} + (L^{(k, k)})^T \\
&- 2 \gamma_k (L^{(k, k)})^T \Sigma L^{(k, k)}) \bdelta_k^\star \nonumber \\
&+ \sum\limits_{j \in S: j \neq k} L^{(k, j)} 
\bdelta_j^\star \\
&- 2 \gamma_k (L^{(k, k)})^T \Sigma \sum\limits_{j \in S: j \neq k} L^{(k, j)} 
\bdelta_j^\star 
\bigg]
\nonumber \\
\Rightarrow \bm{z}^{(k)} &= T^{(k,k)} \bdelta_k^\star + \sum\limits_{j \in S: j \neq k} T^{(k,j)} \bdelta_j^\star,
\label{eq:nash-linear-modified}
\end{align}
where $T^{(p,q)}$ is defined as in Algorithm~\ref{alg:nash} and and $\bm{z}^{(k)} = (2 \gamma_k (L^{(k, k)})^T \Sigma - I)\bm{v}_k$. This proves Eq.~\ref{eq:toprove1}.

Having verified Eq.~\ref{eq:toprove1}, it follows that a Nash equilibrium exists, the modified Algorithm~\ref{alg:nash} finds it. Conversely, a tuple $(\bdelta_i^\star)_{i \in S}$ that solves Eq.~\eqref{eq:nash-linear-modified} for all $k$ is such that $\bdelta_i^\star$ is the optimal $\bdelta_i$ for all $i \in S$ given that other agents report $(\bdelta_j^\star)_{j \in S \setminus \{i\}}$. If no Nash equilibrium exists, Eq.~\eqref{eq:nash-linear-modified} cannot be simultaneously satisfied for all $k$, so the modified Algorithm~\ref{alg:nash} returns ``No Nash Equilibrium.''
\end{proof}

\subsection{Proof of Proposition~\ref{prop:b-recovery}}

For completeness, we first state the techincal result of \cite{torrent-2015} that we require. 

\begin{thrm}[\cite{torrent-2015}]
Let $X \in \RR^{n \times d}$ be a 
design matrix and $C > 0$ an absolute constant. Let $\bb$ be a corruption vector with $\norm \bb \norm_0 \leq \alpha n$, 
$\alpha \leq C$. 
 
Let $\by = X\bw^* + \bb$ be the observed responses, and $\beta \geq \alpha$ be the active set threshold given to the Algorithm 2 of \cite{torrent-2015}. 

Suppose $X$ satisfies the SSC property at level $1 - \beta$ and SSS property at level $\beta$, with constants $\lambda_{1-\beta}$ and and $\Lambda_\beta$ respectively. If the data $(X, \by)$ are such that $\frac{4\sqrt{\Lambda_\beta}}{\sqrt{\lambda_{1 - \beta}}} < 1$, 
then after $t$ iterations, Algorithm 2 of \cite{torrent-2015} with active set threshold $\beta \geq \alpha$ obtains a solution $\bw^t \in \RR^d$ such that
\begin{align*}
\norm \bw^t - \bw^* \norm_2 &\leq \frac{\norm \bb \norm_2}{\sqrt{n}} \exp(-ct)
\end{align*}
for large enough $n$.
\label{thrm:torrent}
\end{thrm}

We are ready to prove Proposition~\ref{prop:b-recovery}. 
\begin{proof}[Proof of Proposition~\ref{prop:b-recovery}]
Let $\vc(\widehat{H^\prime})$ be as in Algorithm~\ref{alg:learning}, and $\bm{y} = \vc(\widehat{H^\prime})$. Notice $\bm{y} = \vc(H) + \vc(H^\prime - H) + \vc(\widehat{H^\prime} - H^\prime)$. Let $\bb := \vc(H^\prime - H)$ be the corruption vector due to strategic negotiations and $\bm{r} = \vc(\widehat{H^\prime} - H^\prime)$ be the residual vector. Recall: 
\[
\vc(\widehat{H^\prime}) = \arg\min_{\bm{v} \in \RR^{n^2}} \norm \vc(W^\prime) - K^{-1} \bm{v} \norm_2^2. 
\]
Since $K^{-1}$ is full rank, $\vc(\widehat{H^\prime}) = K \vc(W^\prime) = \vc(H^\prime)$, so $\bm{r} = 0$. 

Next, we apply Theorem ~\ref{thrm:torrent}. Let $\Tilde{X} = X \otimes X$. Notice that $\norm \bb \norm_0 \leq 2ns - s^2 = \beta$ since $(H^\prime - H)_{i,j}$ is zero if $i, j \not \in S$. Therefore the fraction of corrupted entries is at most $\beta = \frac{2\abs{S}}{n} - \frac{\abs{S}^2}{n^2} \leq 1$. Therefore, if $C$ is the constant in Theorem~\ref{thrm:torrent}, then $\alpha \leq C$ if and only if $\abs{S} \leq C^\prime n$ for some constant $C^\prime$ depending on $C$. 

Further, the design matrix $\Tilde{X}$ satisfies the required SSC and SSS conditions. Therefore, after $T$ iterations, Algorithm~\ref{alg:learning} obtains $\hat{B}$ such that: 
\[
\norm \vc(\hat{B} + \hat{B}^T - (B + B^T)) \norm_2 \leq \frac{\exp(-cT)}{n} \norm \vc(H^\prime - H)\norm_F
\]
\end{proof}

\subsection{Estimating the set of strategic agents}\label{sec:s-recovery}



\begin{proposition}
Under the conditions of Proposition~\ref{prop:b-recovery}, let $b_{min}$ be the least nonzero entry of $(H^\prime - H)$ in absolute value and $T > 0$. Then there exist constants $\rho, C > 0$ such that if $b_{min}$ satisfies: 
\[
\abs{b_{min}} > \exp(-\rho T) \norm \bb \norm_2 
\]
then Algorithm~\ref{alg:learning} with threshold parameter $\beta = \frac{2n\abs{S} - \abs{S}^2}{n^2}$ and $T$ iterations of \textsc{Torrent} recovers $S$ exactly. 
\end{proposition}

\begin{proof}
We proceed by analyzing the residual matrix $R$ of Algorithm~\ref{alg:learning}. 

Let $\eta = \frac{4 \sqrt{\Lambda_\beta}}{\sqrt{\lambda_{1-\beta}}}$. We have $\eta < 1$ by assumption, so let $\rho = 1 - \eta > 0$. Recall that \textsc{Torrent} maintains a set $S_t \subset [n^2]$ called the {\em active set}, which is its guess at iteration $t$ for what indices of the response vector $\vc(H^\prime)$ are non-corrupted. Let $\hat{B}^{(t)} \in \RR^{d \times d}$ be the estimate of \textsc{Torrent} at iteration $t$. Let $\bb^{(t)} = \vc(H^\prime) - \frac{1}{2} (X \otimes X) \vc(\hat{B}^{(t)} + \hat{B}^{(t)})$ be the residual at iteration $t$, and $\bb_{S_{t}} \in \RR^{n^2}$ be the coordinate projection vector such that: 
\[
\bb_{S_{t};i} = 
\begin{cases}
\be_i^T \bb^{(t)} & i \in S_t \\
0 & \text{otherwise}
\end{cases}
\]

From the proof of \cite{torrent-2015} Theorem 10, we obtain that if $S_{t+1}$ is the active set at time $t + 1$, then:
\[
\norm \bb_{S_{t+1}} \norm_2 \leq \eta \norm \bb_{S_t} \norm_2,
\]

Successively applying the inequality and noting that the first estimated active set $S_0 = [n]^2$, we have that: 
\begin{align*}
\norm \bb_{S_{t+1}} \norm_2 &\leq \eta^{t+1} \norm \bb \norm_2 \\
&\leq \exp(-\rho t) \norm \bb \norm_2 
\end{align*}
By assumption on $b_{min}$, the above event can only occur if $\norm \bb_{S_{t+1}} \norm_2 = 0$. Hence $\norm \bb_{S_T} \norm_2 = 0$, so the final active set $S_{T}$ must be a subset of the non-corrupted entries of $\bb$. Hence $R_{ij} = 0$ if $i \not \in S, j \not \in S$. Further, since $S_T$ is the output of a hard-thresholding operation, $\abs{S_T} = (1-\beta)n^2 = 2sn - n^2$. Therefore, after a permutation, the residual matrix is precisely 
\begin{align*}
R &= \begin{bmatrix}
\bm{1}_s \bm{1}_s^T & \bm{1}_s \bm{1}_{n-s}^T \\
\bm{1}_{n-s} \bm{1}_s^T & \bm{0}_{n-s} \bm{0}_{n-s}^T	
\end{bmatrix}.
\end{align*}
A calculation shows that $R$ is rank-two with nonzero eigenvalues $\frac{s(1 \pm \sqrt{(4n/s) - 3})}{2}$. Let $\lambda_2$ be the lesser eigenvalue. The corresponding eigenvector $\bm{v}_2$ has entries $a, b$ at indices $\{1, 2, \dots, s\}$ and indices $\{s+1, \dots, n\}$ respectively, where $a = \frac{1 - \sqrt{(4n / s) - 3}}{2}, b = 1$. Let $S_1, S_2$ be as in Algorithm~\ref{alg:learning}. We have that $S_1 = S$ and $S_2 = [n]\setminus S$ by \cite{rohe2011spectral}. 

Since $\beta = \frac{2 n \abs{S} - \abs{S}^2}{n^2}$, it follows that $\frac{\sqrt{8 \beta n} + 1}{4}$ is closer to $\abs{S}$ than $n - \abs{S}$, so the output $\hat S$ is $S_1 = S$.  
\end{proof}

\section{Analysis of Model Networks}
\label{sec:appendix:modelnetworks}

In this section we will prove Proposition~\ref{prop:3p-insights} and also analyze an additional model network. 

\subsection{Two Agents That Can Self-Invest}
Consider a network with two agents (P1 and P2) who can form a contract with each other, and each can invest in themselves (``self-invest'').
The network setting is as follows. 
\begin{align}
M = \begin{bmatrix}
    M_{11} & M_{12} \\
    M_{21} & M_{22}
\end{bmatrix},
\quad \Sigma = \begin{bmatrix}
    1 & \rho \\
    \rho & 1
\end{bmatrix},
\quad \Gamma = \begin{bmatrix}
    1 & 0 \\
    0 & 1
\end{bmatrix}
\label{eq:2P}
\end{align}
The parameter $\rho \in (-1, 1)$ is the correlation between an agent's returns from self-investment versus trading.
For $\rho\approx 1$, returns from self-investing and trading move in lockstep.
So, each agent must hedge between self-investing and trading, hoping to benefit from any differences in their returns.
But as $\rho$ goes to $-1$, the risks from self-investing and trading offset.
If both offer positive returns, an agent can gain nearly risk-free reward.
Hence, negative correlations can lead to higher utility for agents.

\begin{theorem}
\label{thrm:insights_2p}
Consider the network setting of Eq.~\ref{eq:2P}.
Let $\kappa:=\rho(M_{11} + M_{22}) - (M_{12} + M_{21})$, and let $\Delta$ be the negotiation positions (Definition~\ref{def:negotiation}) for all agents at the Nash equilibrium. 
\begin{enumerate}
\item If only agent $k$ is strategic, then 
\begin{align*}
\Delta_{ij} &= 
  \begin{cases}
  \kappa/3 & \text{if $i\neq k, j=k$}\\
  0 & \text{otherwise}
  \end{cases}
\end{align*}
\item If both agents are strategic, then
\begin{align*}
\Delta_{ij} &= 
  \begin{cases}
  \kappa/4 & \text{ if $i\neq j$}\\
  0 & \text{otherwise}
  \end{cases}
\end{align*}
\end{enumerate}
\end{theorem}
\begin{remark}
We can show that it is strategic for $i$ to report $\Delta_{ii} = 0$ even if $\Sigma_{1;1} \neq \Sigma_{2;2}$ and $\Gamma_{11} \neq \Gamma_{22}$.
\end{remark}

\begin{figure*}[t]
    \begin{subfigure}{\textwidth}
      \centering
      \includegraphics[width=\textwidth]{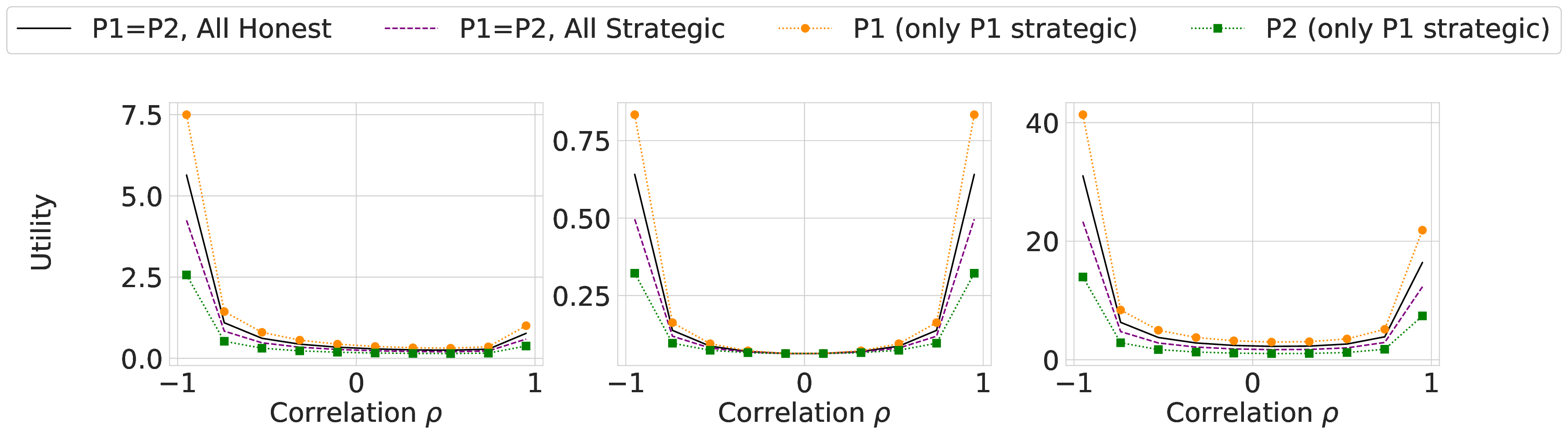}
    \end{subfigure}
    \begin{subfigure}{0.33\textwidth}
      \hspace{2.5em}
      \includegraphics[width=0.9\textwidth]{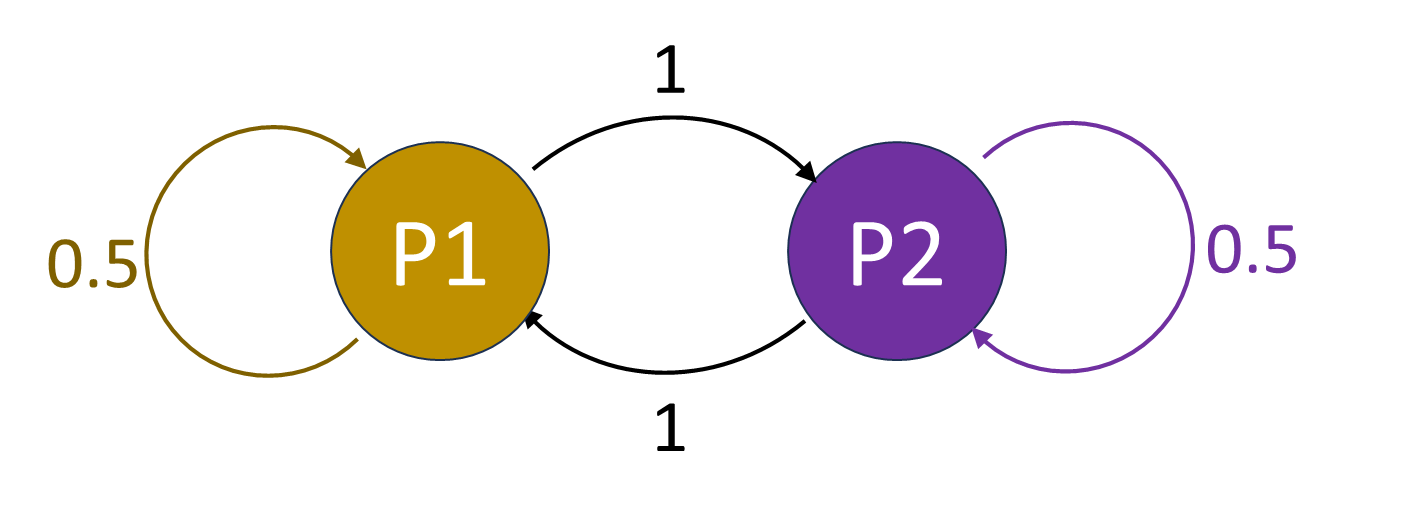}
    \end{subfigure}%
    \begin{subfigure}{0.33\textwidth}
      \hspace{1em}
      \includegraphics[width=0.9\textwidth]{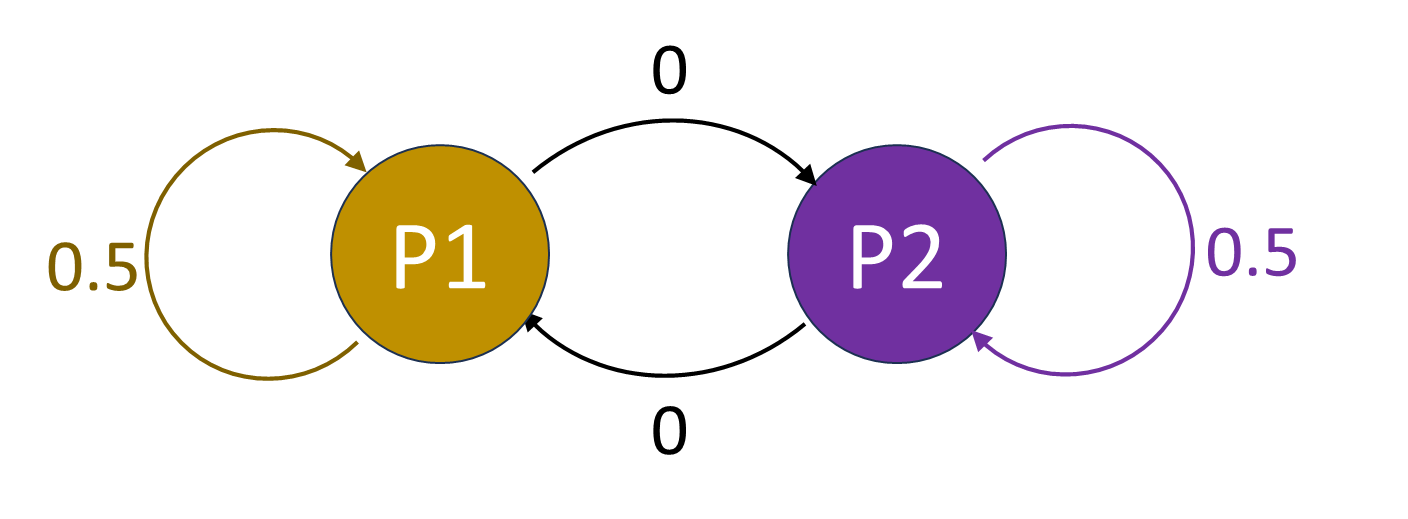}
    \end{subfigure}%
    \begin{subfigure}{0.33\textwidth}
      \hspace{0em}
      \includegraphics[width=0.9\textwidth]{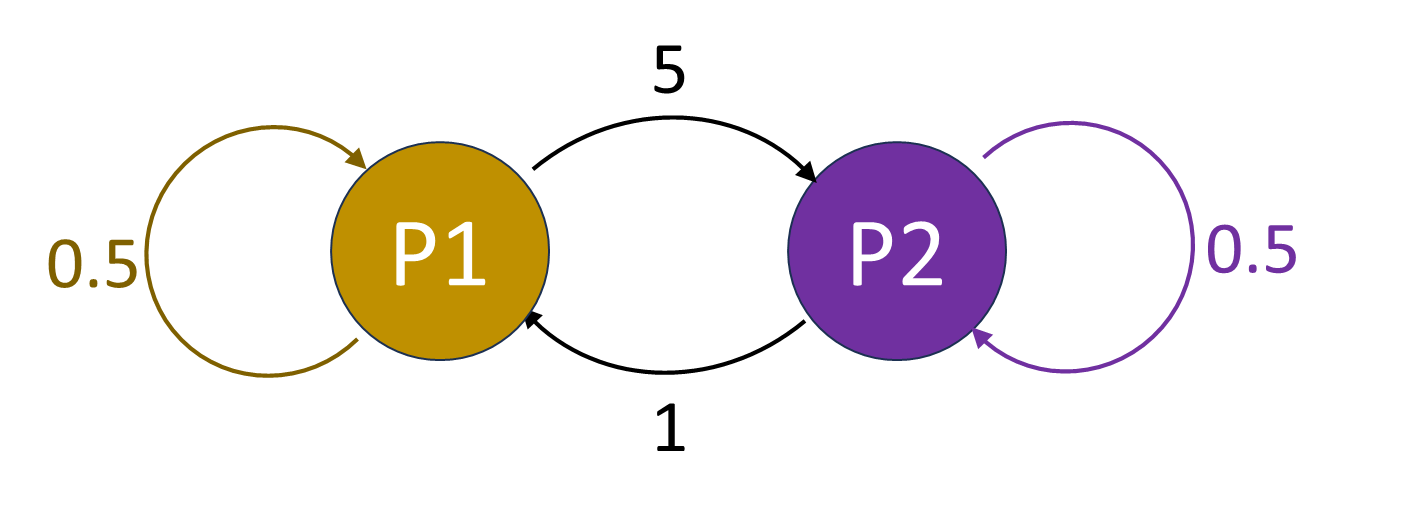}
    \end{subfigure}
    
  \caption{{\em Nash equilibria for two agents:}
  The utility for either agent when both are honest (solid line) is higher than when both are strategic (dashed line).
  When only agent P1 is strategic, P1 gains the highest utility (dotted circles) while P2's utility is lowest (dotted squares).
  The network settings are shown in the bottom row, with an arrow from $i$ to $j$ corresponding to $M_{ji}$.}
  \label{fig:nash-2p}
\end{figure*}

Figure~\ref{fig:nash-2p} shows the agents' utility for honest versus strategic negotiating positions over a range of $\rho$.
We can make several observations.

\smallskip\noindent
{\bf Strategic agents report self-investing returns truthfully.}
Suppose agent $i$ claims that her self-investments have higher returns than in reality (that is, $M^\prime_{ii} > M_{ii}$).
If agent $j$ wants to trade with $i$, then $j$ will have to offer better trading terms via better prices.
Thus, high self-investing returns are a plausible negotiating strategy.
However, Theorem~\ref{thrm:insights_2p} shows that $\Delta_{ii}=0$ at the Nash equilibrium, so $M^\prime_{ii}=M_{ii}$.
This is because if both agents make untrue claims about self-investing, they get smaller contracts, lowering utility.


\smallskip\noindent
{\bf Payments to others can increase when the agent becomes strategic.}
It may appear that strategic agents can only increase their utility by extracting higher payments from others.
However, this need not be true.
Suppose both agents have a utility of $1$ from self-investing and $0$ from trading.
By symmetry, if both agents are honest, they make no payments.
Now, suppose only agent P1 is strategic and $\rho\approx 1$.
By Theorem~\ref{thrm:insights_2p}, P1 will claim to have {\em higher} returns from trading than her actual returns.
This implies that P1 pays P2 during contract formation.
But the contract size also changes.
With the new contract size, P1 still gains utility at the expense of P2.

\smallskip\noindent
{\bf Utility is lower when both agents are strategic.}
Figure~\ref{fig:nash-2p} shows several instances where the agents are worse off when both are strategic versus when both are honest.
This is because the agents face a Prisoner's Dilemma.
If both are honest, they cooperate, and both gain high utility.
However, being strategic is a dominant strategy.
This forces both to be strategic, leading to lower utility for both.


\smallskip\noindent
{\bf Negative correlations amplify the effect of negotiating positions.}
Suppose self-investing and trading both have positive expected returns.
When correlations are negative, their risks cancel while their returns add.
So, an agent can take large positions and achieve high utility.
But, as noted in the previous paragraph, there is a drop in utility when both agents are honest versus strategic.
We find a larger drop for negative correlations.
Hence, under negative correlations, the effect of strategic behavior is also more pronounced.

\subsection{One Investor and Two Hedge Funds (Example~\ref{ex:investor})}\label{sec:3p-no-mechanism}

Consider a $3$-agent network where one investor interacts with two hedge funds under the following network setting.
\begin{align}
M = \begin{bmatrix} 0 & a & a \\ m & 0 & 0 \\ m & 0 & 0 \end{bmatrix}, 
\quad \Sigma = \begin{bmatrix} 1 & 0 & 0 \\
0 & 1 & \rho \\
0 & \rho & 1
\end{bmatrix},
\quad \Gamma = I.
\label{eq:3p-appendix}
\end{align}
The first column corresponds to the investor, and the others to the hedge funds.
Under this setting, the hedge funds do not want to trade with each other, and none of the agents want to self-invest.
Also, the hedge funds are correlated with each other (via $\rho$), and uncorrelated with the investor.

For convenience, we restate Proposition~\ref{prop:3p-insights} below. 

\begin{proposition}[Restatement of Proposition~\ref{prop:3p-insights}]
\label{thrm:3p-insights-appendix}
Consider the the network setting of Eq.~\ref{eq:3p}, where strategic agents can only modify the non-zero entries in their column of $M$.
Define
\begin{align*}
\nu &= \frac{1}{2}\left(\frac{1}{2-\rho} + \frac{1}{2+\rho}\right), &
\eta &= \frac{1}{2}\left(\frac{1}{2-\rho} - \frac{1}{2+\rho}\right), \\
\zeta &= \frac{\nu-\eta}{\nu + (\nu-\eta)(1-\nu)}.
\end{align*}
\begin{enumerate}
\item {\em Honest investor and strategic hedge funds ($S = \{2,3\}$):} 
\begin{align*}
M^\prime_{21}=M^\prime_{31}&= m, &
M^\prime_{12}=M^\prime_{13} &= \frac{a\nu - m (1-\nu)(\nu-\eta)}{\nu + (1-\nu)(\nu-\eta)}.
\end{align*}

\item {\em All agents strategic ($S = [n]$):} 
\begin{align*}
M^\prime_{21}=M^\prime_{31} &= \frac{m - a \zeta}{1 + \zeta}, \\
M^\prime_{12}=M^\prime_{13} &= \frac{a\nu - M^\prime_{21} (1-\nu)(\nu-\eta)}{\nu + (1-\nu)(\nu-\eta)}.
\end{align*}
\end{enumerate}
\end{proposition}

We first discuss insights from Proposition~\ref{prop:3p-insights} and Figure~\ref{fig:insights-3p} of the main text, and then give the proof. 

\smallskip\noindent
{\bf The investor's utility is very sensitive to her negotiating position.}
Suppose the investor is honest and both hedge funds are strategic.
Then, the investor will accept worse terms from the funds and achieve less utility.
But the situation is reversed if the investor is also strategic (Figure \ref{fig:insights-3p}).
The investor now achieves higher utility than either fund.
Thus, the investor's outcome is very sensitive to her negotiating position.

\smallskip\noindent
{\bf The sensitivity to negotiating positions increases as $\rho \to -1$.}
Figure~\ref{fig:insights-3p} shows that when $\rho$ decreases, the investor loses utility if she is honest but gains utility if she is strategic.
The reason is that as $\rho \to -1$, the investor wishes to invest almost equally in both funds to reduce her risk.
The hedge funds only form one contract each.
Since they cannot hedge their risk, they prefer much smaller contracts than the investor.
The investor can extract higher payments for this, increasing her utility.

\smallskip\noindent
{\bf Strategic behavior can reduce utility.}
Suppose the investor is honest.
As $\rho \approx -1$, the hedge funds are worse off being strategic than if they were both honest (Figure~\ref{fig:insights-3p}).
If both funds are honest, their contract sizes match their risk preference.
However, if both are strategic, each fund worries about its competitor.
So, both funds end up with worse terms.

We prove Proposition~\ref{prop:3p-insights}  Part 1 in 
Corollary~\ref{cor:hf-strategic-general}, and Part 2 in Proposition \ref{prop:thrm-3p-insghts-part-2} below. Throughout the remainder of this section, we will refer to the investor as P1 and the hedge funds as P2 and P3. Hence P1 has beliefs according to $M \be_1$, and so on.


\begin{proposition}
Assume that P1 reports $M_{21}, M_{31}$ as $\Tilde m$, P2 reports $M_{12}$ as $\Tilde a$, and P3 reports $M_{13}$ as $\Tilde b$. Then $w_{21} = 0.5 \cdot (\Tilde \alpha + \Tilde a\nu - \Tilde b\eta)$ and $w_{31} = 0.5 \cdot (\Tilde \alpha + \Tilde b\nu - \Tilde a \eta)$ for $\Tilde \alpha = \frac{\Tilde m}{2 + \rho}$. 
\label{prop:3p-contracts-hfs-strategic}
\end{proposition}
\begin{proof}
Notice that $\Sigma$ has eigenvalues $\lambda_1 = 1, \lambda_2 = 1 + \rho, \lambda_3 = 1 - \rho$ and corresponding eigenvectors $\bm{v}_1 = (1, 0, 0)^T, \bm{v}_2 = \frac{1}{\sqrt{2}}(0, 1, 1)^T, \bm{v}_3 = \frac{1}{\sqrt{2}}(0, 1, -1)^T$. Therefore $2^{-1/2}(\bm{w}_2 + \bm{w}_3) = W \bm{v}_2$. Let $\Tilde M$ be the matrix of reported values, so $\Tilde M_{21} = \Tilde M_{31} = \Tilde m$, $\Tilde M_{12} = \Tilde a$, and $\Tilde M_{13} = \Tilde b$.  All other entries of $\Tilde M$ are zero. 

From Theorem \ref{thrm:sylvester-eqns}, the resulting network is $W = \sum\limits_{i, j \in [3]} \frac{\bm{v}_i^T (\Tilde M + \Tilde M^T)\bm{v}_j}{2 (\lambda_i + \lambda_j)} \bm{v}_i \bm{v}_j$. By orthogonality of eigenvectors, we have: 
\[
W \bm{v}_2 = \sum\limits_{i} \frac{\bm{v}_i^T (\Tilde M + \Tilde M^T)\bm{v}_2}{2 (\lambda_i + \lambda_2)} \bm{v}_i 
\]
Only the term at $i = 1$ is nonzero, and therefore $2^{-1/2}(\bm{w}_2 + \bm{w}_3) = \frac{2 \Tilde m + \Tilde a + \Tilde b}{2\sqrt{2}(2 + \rho)}\bme_1$. Similarly, $2^{-1/2}(\bm{w}_2 - \bm{w}_3) = \frac{\Tilde a - \Tilde b}{2\sqrt{2}(2 - \rho)}\bme_1$. The conclusion follows. 
\end{proof}

\begin{proposition}
Let $\Tilde \alpha$ be as in Proposition \ref{prop:3p-contracts-hfs-strategic}. Assume P1 reports $\Tilde \mu$, and P3 reports $\Tilde b$. The optimal choice of reported $\Tilde a$ for P2 is: 
\begin{align*}
\Tilde a^* = c_a + s \Tilde b
\end{align*}
For $c_a = \frac{a \nu - \Tilde \alpha (1 - \nu)}{\nu(2-\nu)}$ and $s = \frac{\eta(1-\nu)}{\nu(2-\nu)}$.
\label{prop:a-opt}
\end{proposition}

\begin{proof}
Let $w := w_{21}$ for shorthand. The utility of firm $2$, by Lemma~\ref{lemma:liar-util}, is given by: 
\begin{align*}
g_2 &= - w (\Tilde a - a) + w^2 \\
&= w(a - \Tilde a + w) \\
2 g_2 &= (\Tilde \alpha + \Tilde a \nu - \Tilde b \eta)( a + (0.5 \nu - 1) \Tilde a + 0.5 \Tilde \alpha - 0.5 \Tilde b \eta) 
\end{align*}
The coefficient of $\Tilde a^2$ in $g_2$ is $\nu (0.5 \nu - 1)$. Since $\nu > 0$ and $0.5 \nu < 1$ for all $\rho \in (-1, 1)$, the Hessian of $g_2$ with respect to $\Tilde a$ is negative definite, and so the optimal choice of $a$ is at the critical point $\frac{\del g_2}{\del \Tilde a} = 0$. Solving for $\Tilde a$ gives: 
\[
\Tilde a^* = c_a + s \Tilde b,
\]
with $c_a, s$ as in the statement of the proposition. 
\end{proof}

A symmetric argument gives the following. 
\begin{proposition}
Let $\Tilde \alpha$ be as in Proposition \ref{prop:3p-contracts-hfs-strategic}. Assume P1 reports $\Tilde \mu$, and P2 reports $\Tilde a$. The optimal choice of reported $\Tilde b$ for P3 is: 
\[
\Tilde b^* = c_b + s \Tilde a
\]
For $c_a = \frac{b \nu - \Tilde \alpha (1 - \nu)}{\nu(2-\nu)}$ and $s = \frac{\eta(1-\nu)}{\nu(2-\nu)}$.
\label{prop:b-opt}
\end{proposition}

Next, we can solve for the Nash equilibria given the reported $\Tilde m$ of the investor. 
\begin{proposition}
If $M_{12} = M_{13} = a$, then let $c:= c_a = c_b$ and $s$ be as in Proposition \ref{prop:a-opt} and \ref{prop:b-opt}. Assume P1 reports $\Tilde \mu$. The Nash equilibrium for P2, P3 is to report:  
\begin{align*}
M_{12}^\prime = M_{13}^\prime &= \frac{c}{1 - s} 
\end{align*}
\label{prop:nash-3p-naive-investor}
\end{proposition}

\begin{corollary}[Proposition~\ref{prop:3p-insights} Part 1]
Assume P1 reports $\Tilde m$. If both hedge funds are strategic, then the Nash equilibrium for P2, P3 is to report: 
\begin{align*}
M_{12}^\prime = M_{13}^\prime &=  \frac{a \nu - \Tilde m (1 - \nu)(\nu - \eta)}{\nu + (1-\nu)(\nu - \eta)} 
\end{align*}
Hence if $\Tilde m = m$, then $M_{12}^\prime = M_{13}^\prime$ are as in Proposition~\ref{prop:3p-insights}.1.
\label{cor:hf-strategic-general}
\end{corollary}

\begin{proof}
We simplify $a_{NS} = \frac{c}{1 - s}$ as follows. 
\begin{align*}
\frac{c}{1 - s} &= \frac{a \nu - \Tilde \alpha (1 - \nu)}{\nu (2 - \nu)(1 - s)} \\
&= \frac{a \nu - \Tilde \alpha (1 - \nu)}{\nu (2 - \nu)\big(1 - \frac{\eta(1-\nu)}{\nu (2-\nu)}\big)} \\
&= \frac{a \nu - \Tilde m (1 - \nu)(\nu - \eta)}{\nu(2-\nu) - \eta(1-\nu)} \\
&= \frac{a \nu - \Tilde m (1 - \nu)(\nu - \eta)}{\nu + (1-\nu)(\nu - \eta)} 
\end{align*}
In the setting of Proposition~\ref{prop:3p-insights}.1, the investor reports $\Tilde m = m$. The conclusion follows. 
\end{proof}

Next, we solve for the optimal report of the investor if all agents are strategic. 

\begin{proposition}[Proposition~\ref{prop:3p-insights}  Part 2]
Let $y = \frac{1}{2(2+\rho)}$. If all agents are strategic, then the optimal reported $\Tilde m$ for the investor is: 
\[
M_{21}^\prime = M_{31}^\prime = \frac{m - a \zeta}{1 + \zeta}
\]
For $\zeta = \frac{\nu - \eta}{\nu + (\nu - \eta)(1-\nu)}   = (1 - 2y(1+\rho))$. 
The optimal report for the hedge funds is: 
\[
M_{12}^\prime = M_{13}^\prime = \frac{a \nu - M_{21}^\prime (1-\nu)(\nu - \eta)}{\nu + (1-\nu)(\nu - \eta)}
\]
\label{prop:thrm-3p-insghts-part-2}
\end{proposition}

\begin{proof}
From Proposition \ref{prop:nash-3p-naive-investor} and \ref{prop:3p-contracts-hfs-strategic}, we know $w_{12} = w_{13}$. Therefore by Lemma \ref{lemma:liar-util}, if P1 reports $\Tilde m$ then the investor utility is:  
\begin{align*}
g_1(\Tilde m) &= 2 (m - \Tilde m) w_{12} + w_{12}^2 (2 + 2\rho)
\end{align*}
Let $w_{12} = (c_2 + \Tilde m y)$ for shorthand. Then $g_1$ is quadratic in $\Tilde m$ and the coefficient of $\Tilde m^2$ is $2y^2 + 2 \rho y^2 - 2y = \frac{-(\rho +3)}{2(\rho + 2)^2} < 0$. Therefore, the optimal $\Tilde m$ is at the critical point $\frac{\del g_1}{\del \Tilde m} = 0$. This is given as: 
\begin{align*}
\Tilde m &= \frac{-c_2 + m y + 2c_2 y(1+\rho)}{2y(1 - y(1 + \rho)} \\
&= \frac{m - (c_2/y) (1 - 2y(1+\rho))}{2(1 - y(1+\rho))} \\
&= \frac{m - (c_2/y) (1 - 2y(1+\rho))}{1 + (1 - 2y(1+\rho))} 
\end{align*}
We simplify the terms $(c_2 / y), 2y(1+\rho)$ as follows. First, notice that $c_2 = \frac{\nu - \eta}{2} \cdot \frac{a \nu}{\nu(1-\nu)(1-s)}$, where $a$ is true value of $M_{12}$ and $M_{13}$ for the hedge funds. Let $c_1 := \frac{a \nu}{\nu(1-\nu)(1-s)}$ for shorthand, so that $c_2 = \frac{\nu - \eta}{2} c_1$. Next, 
\begin{align*}
\frac{c_2}{y} &= \frac{c_1}{\nu - \eta} \\
&= \frac{a}{(2 - \nu)(1-s)} \cdot \frac{2(\nu - \eta)}{2 (2 + \rho)^{-1} (1 - (\nu - \eta)x)} \\
&= \frac{a}{(2 - \nu)(1-s)} \\
&\cdot \bigg(
1 - (\nu - \eta) \cdot \frac{1-\nu}{\nu(2-\nu)(1-s)}
\bigg)^{-1} \\
&= \frac{a \nu}{\nu(2-\nu)(1-s) - (\nu - \eta)(1 - \nu)} \\
&= \frac{a \nu}{\nu(2-\nu) - \eta(1-\nu) - (\nu - \eta)(1-\eta)} \\
&= \frac{a\nu}{\nu} \\
&= a \\
1 - 2y(1+\rho) &= 1 - 2(c_2 / a)(1 + \rho) \\
&= 1 - (1 + \rho)\frac{c_1(\nu - \eta)}{a} \\
&= 1 - (1 + \rho)\cdot \frac{a}{(2-\nu)(1-s)}\frac{(\nu - \eta)}{a} \\
&= 1 - \frac{(\nu-\eta)(1+\rho)\nu}{\nu(2-\nu) - \eta(1-\nu)} \\
&= 1 - \frac{\nu(\nu-\eta)(1+\rho)}{\nu + (\nu - \eta)(1 - \nu)} \\
&= 1 - \frac{\nu(\nu-\eta)((2 + \rho) - 1)}{\nu + (\nu - \eta)(1 - \nu)} \\
&= 1 - \frac{\nu - \nu(\nu - \eta)}{\nu + (\nu - \eta)(1 - \nu)} \\
&= \frac{\nu - \eta}{\nu + (\nu - \eta)(1-\nu)} 
\end{align*}
Let $\zeta := \frac{\nu - \eta}{\nu + (\nu - \eta)(1-\nu)}  = 1 - 2y(1+\rho)$. The optimal $\Tilde m^* = \frac{m - a \zeta}{1 + \zeta}$. 

Finally, substituting this $\Tilde m^*$ into Corollary \ref{cor:hf-strategic-general} gives the optimal reports for the hedge funds. 

\end{proof}

\section{Additional Experiments}

In this section, we describe additional results from negotiations on international trade networks. We compute analogous results to Figure~\ref{fig:oecd} of the main text, and find that the results are qualitatively similar when the US (Figure~\ref{fig:oecd-us}) and when India (Figure~\ref{fig:oecd-india}) are the lone strategic actors. 

We find that the Netherlands is worst-off when all are strategic, and best off if all are honest. Moreover, each strategic actor is best-off when they are the lone strategic actor. These results are the same as in the case of the UK being the strategic actor, shown in Figure~\ref{fig:oecd} of the main text. 

\begin{figure}[!ht]
    \centering
    \includegraphics[width=0.49\textwidth]{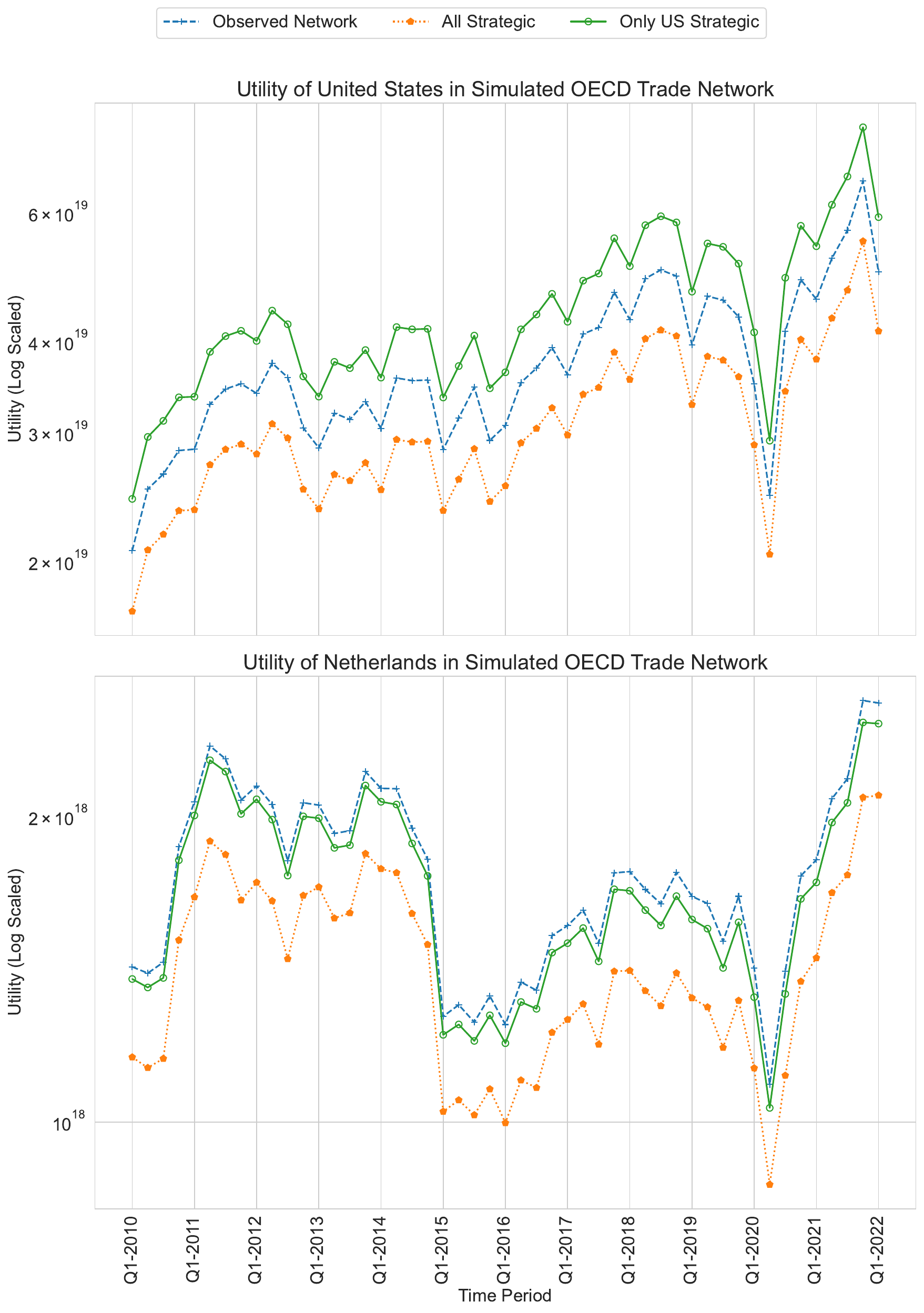}
    \caption{{\em Effect of strategic trading on the US (top) and Netherlands (bottom):}
    The US's utility is highest. 
    The Netherlands has the highest utility when all others are honest, unlike the UK.
    }
    \label{fig:oecd-us}
\end{figure}

\begin{figure}[!ht]
    \centering
    \includegraphics[width=0.49\textwidth]{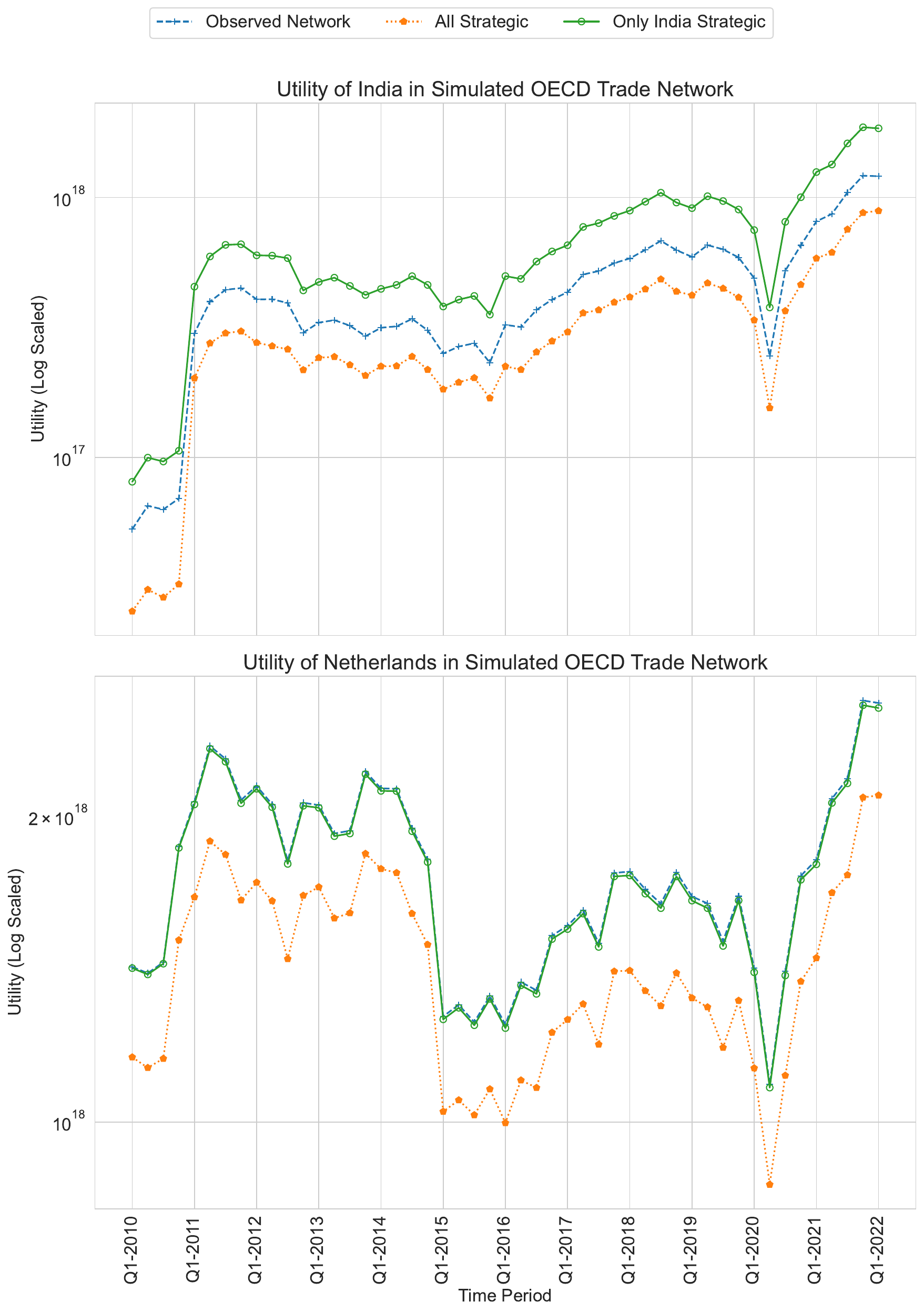}
    \caption{{\em Effect of strategic trading on India (top) and Netherlands (bottom):}
    India's utility is highest. 
    The Netherlands has the highest utility when all others are honest, unlike the UK.
    }
    \label{fig:oecd-india}
\end{figure}

\section{Experimental Details}\label{appendix:experiments}

\subsection{Learning Experiments}

In all learning experiments, we generate random $X, B, \Sigma$ as follows. $X \in \RR^{n \times d}$ has rows that are iid Dirichlet$(1/d, \dots, 1/d)$. $B \in \RR^{d \times d}$ is symmetric, with $B_{ij} \sim N(5, 1)$ for $i \geq j$ and $B_{ji} = B_{ij}$ for $i < j$. $\Sigma = X B_\Sigma X^T + \eps I_{n}$ for random $B_\Sigma \in \RR^{d \times d}$ and $\eps = 10^{-3}$. The $B_\Sigma$ is generated as $B_\Sigma = U D U^T$ for $U \in \RR^{d \times d}$ having $d$ random orthonormal columns, and $D$ diagonal with iid Uniform$(1, \sqrt{n})$ entries on the diagonal. 

Given the network setting $(B, I, \Sigma, X)$ with some number of strategic agents $s \in [n]$, we generate random $S \subset [n]$ of size $s$, uniformly at random from all subsets of size $s$. Then, we generate negotiating positions $M^\prime$ based on Algorithm~\ref{alg:nash}.

Our impelementation of spectral clustering is as follows. We compute the eigenvector $\bm{v}_2 \in \RR^n$ corresponding to the least nonzero eigenvalue $\lambda_2$ of $R$, and then assigns $S_1, S_2$ to the positive and negative indices of $\bm{v}_2$ respectively. We solve for $s$ from $\beta, n$, and then return $\hat S = S_1$ if $\abs{S_1} > s$ and $\hat S = S_2$ otherwise.

\subsection{Negotiations on International Trade Networks}

We use the same international trade dataset as \cite{jalan-2023}. Specifically, we use international trade statistics from the OECD to get quarterly measurements of bilateral trade between 46 large economies, including the top 15 world nations by GDP OECD (2022) \cite{jalan-2023}. The data are available at the OECD Statistics webpage (\url{https://stats.oecd.org/}). The data are measured quarterly from Q1 2010 to Q2 2022. We take the sum of trade flows $i \to j$ and $j \to i$. The diagonals $W_{ii} = 0$ for all $i$. 

Given these measurements, which we denote as $W_{t} \in \RR^{46 \times 46}$ for $t = 1, 2, \dots, 46$, we solve for a covariance matrix $\Sigma \succ 0$ using the Semidefinite Programming algorithm of \cite{jalan-2023}. This $\Sigma$ is fixed and used throughout the experiments of Section~\ref{sec:experiments}.

\subsection{Compute Environment}

All experiments were performed on a Linux machine with 48 cores running Ubunutu 20.04.4 OS. Each CPU core is an Intel(R) Xeon(R) CPU E5-2695 v2 (2.40GHz). The architecture was x86. Total RAM was 377 GiB. 

The total time to generate experimental results was approximately 24 hours of wall time with parallelization. 
